\documentclass[12pt]{article}
\usepackage[a4paper, left=2cm, right=2cm, top=2 cm, bottom= 2 cm]{geometry} 
\usepackage{graphicx} 

\usepackage{booktabs} 
\usepackage{array} 
\usepackage{paralist} 
\usepackage{verbatim} 
\usepackage{subfig} 

\makeatletter
\def\blfootnote{\xdef\@thefnmark{}\@footnotetext}
\makeatother

\usepackage{amsthm} 

\theoremstyle{plain}
\newtheorem{theorem}{Theorem}

\newtheorem{lemma}[theorem]{Lemma}

\newtheorem{proposition}[theorem]{Proposition}

\theoremstyle{definition}

\newtheorem{definition}[theorem]{Definition}
\newtheorem{remark}[theorem]{Remark}

\usepackage{enumerate} 

\usepackage{amssymb} 

\usepackage{amsfonts,amssymb,url}
\usepackage{graphicx}
\usepackage{epstopdf}
\usepackage{amsmath} 

\usepackage{mathrsfs} 

\usepackage{hyperref}
\usepackage{color} 

\newcommand{\R}{\mathbb{R}}
\newcommand{\N}{\mathbb{N}}

\newcommand{\Z}{\mathbb{Z}}
\newcommand{\T}{\mathbb{T}}

\newcommand{\E}{\mathbb{E}}

\newcommand{\eps}{\varepsilon}

\newcommand{\Td}{\mathbb{T}^d}

\newcommand{\jbrac}[1]{\left\langle {#1}\right\rangle}
\newcommand{\norm}[1]{\left\Vert {#1}\right\Vert}

\DeclareFontFamily{U}{mathx}{\hyphenchar\font45}
\DeclareFontShape{U}{mathx}{m}{n}{
	<5> <6> <7> <8> <9> <10>
	<10.95> <12> <14.4> <17.28> <20.74> <24.88>
	mathx10
}{}
\DeclareSymbolFont{mathx}{U}{mathx}{m}{n}
\DeclareFontSubstitution{U}{mathx}{m}{n}
\DeclareMathAccent{\widecheck}{0}{mathx}{"71}
\DeclareMathAccent{\wideparen}{0}{mathx}{"75}

\title{\Large{Quasi-invariance of fractional Gaussian fields\\ by the nonlinear wave equation with polynomial nonlinearity}
}
 
 \author{\normalsize{Philippe Sosoe}\\ \small{Cornell University} \and \normalsize{William J.~Trenberth} \\ \small{University of Edinburgh} \and \normalsize{Tianhao Xian} \\ \small{Cornell University}}

\begin{document}

\maketitle

\blfootnote{2010 \emph{Mathematics Subject Classification}. 35L71; 60H15.}

\abstract{We prove quasi-invariance of Gaussian measures $\mu_s$ with Cameron-Martin space $H^s$ under the flow of the defocusing nonlinear wave equation with polynomial nonlinearities of any order in dimension $d=2$ and sub-quintic nonlinearities in dimension $d=3$, for all $s>5/2$, including fractional $s$. This extends work of Oh-Tzvetkov and Gunaratnam-Oh-Tzvetkov-Weber who proved this result for a cubic nonlinearity and $s$ an even integer. The main contributions are a modified construction of a weighted measure  adapted to the higher order nonlinearity, and an energy estimate for the derivative of the energy replacing the integration by parts argument introduced in previous works.  We also address the question of (non) quasi-invariance for the dispersionless model raised in the introductions to \cite{OT, GOTW}.}

\section{Introduction}
\subsection{Statement of results}
We consider the nonlinear wave equation (NLW), on $\T^d$ for $d=2$ and $d=3$. This is the following equation for an unknown function $u:\mathbb{T}^d\times \mathbb{R}\rightarrow \mathbb{R}$:
\begin{equation}\label{eqn: pre-NLW}
\begin{cases}
\partial^2_tu-\Delta u+u^k=0\\
(u,\partial_tu)|_{t=0}=(u_0,v_0)
\end{cases}
\end{equation}
\noindent
where $k$ is a positive, odd integer. We rewrite \eqref{eqn: pre-NLW} as a first order system:
\begin{equation}\label{NLW}
\begin{cases}
\partial_t u=v\\
\partial_tv=\Delta u-u^k=0\\
(u,v)|_{t=0}=(u_0,v_0)
\end{cases}.
\end{equation}
\noindent
The system (\ref{NLW}) is a Hamiltonian system with Hamiltonian
\begin{equation}\label{eqn: H}
E(u,v):=\frac{1}{2}\int_{\Td}\left( |\nabla u|^2+v^2\right)\,\mathrm{d}x+\frac{1}{k+1}\int_{\Td}u^{k+1}\,\mathrm{d}x.
\end{equation}
The solutions of the system \eqref{NLW} we work with have initial data in the Sobolev space
\[\vec{H}^\sigma(\mathbb{T}^d):= H^\sigma(\mathbb{T}^d)\times H^{\sigma-1}(\mathbb{T}^d)\]
for $d=2,3$.
We consider the transport properties of the Gaussian measure on initial data $\vec{u}=(u,v)$ formally given by 
\begin{equation}\label{eqn: formal-meas}
\vec{\mu}_s(\mathrm{d}\vec{u})=Z_s^{-1}e^{-\frac{1}{2}\|(u,v)\|^2_{\vec{H}^{s+1}}}\mathrm{d}\vec{u}.
\end{equation}
The expression \eqref{eqn: formal-meas} is given meaning as a product measure on the Fourier coefficients of the pair $(u,v)$:
\begin{equation}
\vec{\mu}_s(\mathrm{d}u\mathrm{d}v)=Z_s^{-1}\prod_{n\in\mathbb{Z}^3}e^{-\frac{1}{2}\langle n\rangle^{2(s+1)}|\widehat{u}_n|^2}e^{-\frac{1}{2}\langle n\rangle^{2s}|\widehat{v}_n|^2}\,\mathrm{d}\widehat{u}_n\mathrm{d}\widehat{v}_n.
\end{equation}
Equivalently, $\mu_s$ is the distribution of the pair of function-valued random variables given by $\omega \mapsto (u^{\omega}, v^{\omega})$, where $u^{\omega}, v^{\omega}$ are the functions on $\mathbb{T}^d$ defined by
\begin{equation}\label{series}
\begin{split}
    u^{\omega}(x) &:= \sum_{n\in \mathbb{Z}^3}\dfrac{g_n(\omega)}{\langle n\rangle^{s+1}}e^{i n\cdot x},\\
 \quad v^{\omega}(x) &:=\sum_{n\in \mathbb{Z}^3}\dfrac{h_n(\omega)}{\langle n\rangle^s}e^{i n\cdot x}.
\end{split}
\end{equation}
Here, $g_n$, $h_n$, $n\in\mathbb{Z}^d$ are standard complex Gaussian random variables, independent except for the conditions:
\begin{equation}\label{eqn: g_n-conditions}
g_n=\bar{g}_{-n},\quad h_n=\bar{h}_{-n}, n\neq 0,
\end{equation}
and $g_0$, $h_0$ real-valued. By inspection of the series \eqref{series}, it is clear that $\mu_s$ is supported on $\vec{H}^{\sigma}$ for $\sigma < s-\frac{d-2}{2}$.

Let $s\ge 1$ and $d=2$ or $d=3$. The system \eqref{NLW} is said to be globally well-posed in $\vec{H}^s$ if, for every $(u,v)\in \vec{H}^s$, there is a solution $(u(t),v(t)=\partial_t u(t)):\mathbb{R}_+\rightarrow \vec{H}^s$ to the integral equation
\[
u(t)=\cos(t\sqrt{-\Delta})(u_0)+\frac{\sin(t\sqrt{-\Delta})}{\sqrt{-\Delta}}(v_0)-\int_0^t \frac{\sin((t-s)\sqrt{-\Delta})}{\sqrt{-\Delta}}u^{k}(s)\,\mathrm{d}s\]
which is unique in the class $C(\mathbb{R}_+;H^1(\mathbb{T}^d)\times L^2(\mathbb{T}^d))$. The Fourier multipliers $\cos(\sqrt{-\Delta})$ and $\frac{\sin(\sqrt{-\Delta})}{\sqrt{-\Delta}}$ are defined by
\begin{align*}
\mathcal{F}(\cos(\sqrt{-\Delta})u)(n)&=\cos(|n|)\widehat{u}(n),\\
\mathcal{F}\big(\frac{\sin(\sqrt{-\Delta}}{\sqrt{-\Delta}}u\big)(n)&=\frac{\sin(|n|)}{|n|}\widehat{u}(n),
\end{align*}
where for $u\in \mathcal{D}(\mathbb{T}^d)$, we use both $\mathcal{F}(u)(n)$ and $\widehat{u}(n)$ interchangeably to denote the $n$th Fourier coefficient:
\[\mathcal{F}(u)(n)=\widehat{u}(n)=\frac{1}{(2\pi)^d}\int_{\mathbb{T}^d} e^{-in\cdot x} u(x)\,\mathrm{d}x.\]
For $d=2$, the system is globally well-posed for any odd integer $k$, for a certain range of regularities. For $d=3$, global wellposedness is known for $k=3$ and $k=5$ for certain regularities. A pedagogical proof of well-posedness in the three-dimesnional case appears in \cite{tzvetkov-notes}.

\begin{theorem}\label{thm: quasi-invariance}
Let $s>\frac{5}{2}$. Let $k$ be an odd integer such that $3\le k <\infty$ if $d=2$ and $k=3$ or $k=5$ when $d=3$.

For any time $t>0$, the distribution of the solution $\Phi(t)(u,v)=(u(t),v(t))$ of the system \eqref{NLW} is absolutely continuous with respect to the distribution $\mu_s$ of the initial data, given by \eqref{series}.
\end{theorem}

The special case of Theorem \ref{thm: quasi-invariance} for $d=2$, $k=3$ and $s$ an even integer was proved by Oh and Tzvetkov in \cite{OT}. This work first introduced the renormalized weighted measure to improve the energy estimate on the support of the Gaussian measure. Gunaratnam, Oh, Tzvetkov and Weber \cite{GOTW} extended the result to dimension 3 ($k=3$, $s\ge 4$). Their proof uses the recent variational formulation of partition function introduced by Barashkov and Gubinelli for the purpose of renormalization of $\phi^4$ field theories, combined with a recent argument of Planchon, Tzvetkov and Visciglia \cite{PTV} for proving quasi-invariance in ``local'' situations, exploiting deterministic growth bounds on the solution.

The somewhat surprising aspect of \cite{GOTW} is that although the weighted measure involves the quartic quantity
\[\int (D^s u_N)^2 u_N^2,\] no renormalization other than the Wick type subtraction \eqref{eqn: Q-def} is necessary, in contrast to the $\phi^4$ model in dimension 3.

Our estimates also yield a result in the spirit of \cite[Theorem 1.5]{PTV} concerning the transport of bounded subsets of $\vec{H}^{\sigma}$ in settings where global well-posedness is not known.
\begin{theorem}\label{thm: bound}
  Let $s>5/2$, and $\sigma<s-\frac{1}{2}$ be sufficiently close to $s-\frac{1}{2}$ for $d=2$, and $s>3$, $\frac{3}{2}<\sigma<s-\frac{1}{2}$ sufficiently close to $s-\frac{1}{2}$ for $d=3$.

For each $R>0$, let
\[B_R(0;\vec{H}^\sigma):=\{(u,v):\|(u,v)\|_{\vec{H}^\sigma}<R\}\]
denote the ball of radius $R$ centered at the origin in $\vec{H}^\sigma$. There exists $T=T(d,R)$ and $C(R)>0$ such that for $(u_0,v_0)\in B_R(0;\vec{H}^\sigma)$,
there is a solution $(u(t),v(t))$ of \eqref{NLW} such that
\begin{equation}\label{eqn: det-bound}
  \sup_{[-t,t]\le T}\|(u,v)(t)\|_{\vec{H}^\sigma}<C(R).
\end{equation}

Let $A$ be a Borel subset of $\vec{H}^{\sigma}$ such that
\[A\subset \{u\in \vec{H}^\sigma: \|u\|_{\vec{H}^\sigma}<R\}\]
and $\vec{\mu}_s(A)=0$. Then
\[\vec{\mu}_s(\{u(t): u(t) \text{ solution of  with initial data } u_0\in A\})=0\]
for all $t\in [-T,T]$.
\end{theorem}
We remark that the assumptions on $s$ and $\sigma$ in the statement of this result are not optimal. For example, for $d=3$, the restrictions on $\sigma$ imply that $\vec{H}^\sigma$ is an algebra, so that the nonlinearity can be treated by the basic tame estimate
\[\|u^k\|_{H^{\sigma-1}(\mathbb{T})}\lesssim \|u\|_{H^{\sigma-1}(\mathbb{T}^3)}\|u\|^{k-1}_{L^\infty(\mathbb{T}^3)},\] 
and so the existence is classical in this case. It is well known that a basic short-time well-posedness result can be proved in the range $s> \frac{3}{2}-\frac{1}{k-1}$ using Strichartz estimates. Since our methods do not reach the optimal range of $s$ in the energy estimate, we do not attempt to optimize $s$ in Theorem \ref{thm: bound}. As remarked in \cite{OT}, \cite{GOTW} it is of interest to consider quasi-invariance in low regularity settings.

Our final result concerns the necessity of the dispersion for the results in Theorem \ref{thm: quasi-invariance} and Theorem \ref{thm: bound} to hold. Omitting the Laplacian term in \eqref{NLW} leads, for any initial data $(u_0,v_0)$, to an ordinary differential system whose solution exists globally thanks to the Hamiltonian structure. Moreover, considering the $x$ dependence, the algebra property of the Sobolev space $\vec{H}^\sigma$ implies that the solution remains regular for all positive times if $\sigma>3$. The next result shows that, unlike preservation of Sobolev regularity, the absolute continuity statements in the two previous results depend crucially on the dispersion.
\begin{theorem}\label{thm: dispersionless}
For $t>0$, let $(\tilde{u}(t),\tilde{v}(t))$ be the solution at time $t$ of the system \eqref{eqn: ODE} with initial data $(u(0),v(0))=(u^\omega,v^\omega)$ distributed according to the random series \eqref{series}. Then $(\tilde{u}(t),\tilde{v}(t))$ is not absolutely continuous with respect to $(u(0),v(0))$.
\end{theorem}
The analogue of Theorem \ref{thm: dispersionless} for a Schr\"odinger-type equation in dimension $d=1$ was proven by Oh, Tzvetkov, and the first author in \cite{OST}. The difference in the nonlinear wave case is that we do not have an explicit solution of the the ODE that appears when the Laplacian term $\Delta u$ is left out of \eqref{NLW}. The proof of Theorem \ref{thm: dispersionless} appears in Section \ref{sec: dispersionless}.

\subsection{Motivation and previous literature}
The current work is motivated by the results of Oh-Tzvetkov \cite{OT} and Gunaratnam-Oh-Tzvetkov-Weber \cite{GOTW} on the transport of Gaussian measures by the flow of the 2d, respectively 3d, cubic nonlinear wave equations. In particular, we address a number of questions mentioned in the introduction to these papers.

The study of the transport properties of Gaussian measures by Hamiltonian dispersive dynamics was recently initiated by N. Tzvetkov in \cite{T1}. The initial motivation in this work was the study of long term estimates in high Sobolev norms. Another motivating question is the existence of invariant measures supported on Sobolev spaces of regularity higher than the formal Gibbs measures.


The paper \cite{T1} follows a long line of work on the dynamics of Hamiltonian dispersive equations with random initial data. This goes back at least to the foundational paper \cite{LRS} by Lebowitz, Rose and Speer (LRS), which constructs measures absolutely continuous with respect to circular Brownian motion. These were expected to be invariant under the flow of the nonlinear Schr\"odinger equation on the torus. J. Bourgain used his $X^{s,b}$ spaces to construct dynamics on the support of these measures and gave a proof of invariance \cite{bourgain1}. Bourgain then applied his method to a number of other equations \cite{bourgain2, bourgain3, bourgain4}.

Quasi-invariance does not have the dynamical implications of invariance, but it is nonetheless a delicate property of the flow, implying for example the propagation of fine regularity properties of the initial data. It is a much stronger property than persistence of the Sobolev regularity implied by the usual well-posedness results. Indeed, it was noted in \cite{OST} that quasi-invariance implies propagation of the (a.s. constant) modulus of continuity of the Gaussian initial data at any point, and this fact was used to show that the dispersion in the nonlinear Schr\"odinger equation is essential for quasi-invariance to hold.

While several results after \cite{T1} used modified measures to obtain a favorable energy estimate, the paper \cite{OT} was the first to consider the addition to the energy of a formally ``infinite'' term. Namely, Oh-Tzvetkov modify the base measure by adding a term which does not converge as the Fourier cutoff is removed on the support of the base Gaussian measure, requiring a renormalization. The renormalization of the modified energy there is based on argument akin to Nelson's argument for the construction of $P(\phi)_2$ quantum fields \cite{N}.

 \paragraph{This paper's contributions.}
 The main differences with previous works, in particular \cite{GOTW}, are as follows:
\begin{enumerate}
\item The key energy estimates in \cite{OT, GOTW} depend on an initial integration by parts (see \cite[Equations (3.5)-(3.6)]{GOTW}). This  integration by parts removes the most singular term in the derivative of the energy.  It is also in this step of the argument that the correction needed to define the weighted measure is identified. 

As pointed out in \cite{GOTW}, when $s$ is not an integer, we cannot integrate by parts and obtain exact cancellation. The main tools here are paraproducts and an expansion of the relevant multiplier. Since we do not require $s$ to be an integer, we automatically lower the accessible regularities for the measure $\vec{\mu}_s$.

\item To construct the weighted measure $\vec{\rho}_s$, one needs uniform control of the partition function of the trunctated measures $\vec{\rho}_{s,N}$. As in \cite{GOTW}, we achieve this by pathwise bounds on the terms in a stochastic optimization problem involving a measure perturbed by a ``control drift''. In our case, the relevant expression involves higher powers of the drift. To ensure positivity, we must modify the weighted measure to incorporate a high power of the energy.
\end{enumerate}

\subsection{Outline of proof of Theorems \ref{thm: quasi-invariance} and \ref{thm: bound}}
Our proof proceeds along the lines of that in \cite{OT,GOTW}, following a general methodology introduced in \cite{T1}. Tzvetkov method is based on the construction of a measure $\vec{\rho}_s$ which is mutually absolutely continuous with respect to the Gaussian measure $\vec{\mu}_s$ of interest, but for which the time evolution of sets can be controlled effectively. Then, we need a suitable energy estimate on the support of the renormalized measure.

We start by replacing the measure $\mu_s$ by an equivalent measure more suitable to the analysis of the nonlinear wave flow.
\begin{definition}
Let $g_n$, $h_n$, $n\in\mathbb{Z}^3$ be standard complex Gaussian random variables satisfying \eqref{eqn: g_n-conditions} such that $g_0$ and $h_0$ are  real valued. Define $\mathrm{d}\vec{\nu}_s$ to be the distribution of the random series
  \begin{align}
    u^\omega(x) &= g_0+ \sum_{n\in \mathbb{Z}^3\setminus \{0\}}\frac{g_n(\omega)}{(|n|^2+|n|^{2s+2})^{\frac{1}{2}}}e^{i n\cdot x}, \label{eqn: useries}\\
    v^\omega(x)&= \sum_{n\in\mathbb{Z}^3}\frac{h_n(\omega)}{(1+|n|^{2s})^{\frac{1}{2}}}e^{in\cdot x}. \label{eqn: vseries}
    \end{align}
\end{definition}

For each $N\ge 1$,  let $\Phi_N(t)$ denote the time $t$ flow of the following approximation of the flow of the equation \eqref{NLW}:
\[\begin{cases}
    \partial_t u =v\\
    \partial_t v = \Delta u -\pi_N((\pi_N u)^k)\\
    (u,v)|_{t=0}= (u_0,v_0),
  \end{cases}\]
where by $\pi_N$ we denote the projection \eqref{eqn: dirichlet-proj} onto frequencies less than $N$.

A change of variables formula (see \cite[Proposition 4.1]{T1} or \cite[Lemma 5.1]{OT}) then implies
\begin{equation}\label{eqn: mu-COV}
\int_{\Phi_N(t)(A)}\vec{\nu}_s(\mathrm{d}\vec{u})=Z_{s,N}^{-1}\int_{A}e^{-\frac{1}{2}\|\Phi_N(t)(\pi_N u)\|^2}\mathrm{d} L_N\otimes \vec{\nu}_{s,N}^\perp(\mathrm{d}\vec{u}).
\end{equation}
In \eqref{eqn: mu-COV}, we have used the notation
\[\|(u,v)\|^2=\int_{\mathbb{T}^3}(D^s v)^2\,\mathrm{d}x+\int_{\mathbb{T}^3}(D^{s+1} u)^2\,\mathrm{d}x +\int_{\mathbb{T}^3} |\nabla u|^2\,\mathrm{d}x+\int_{\mathbb{T}^3} v^2\,\mathrm{d}x+\left(\int_{\mathbb{T}^3} u\,\mathrm{d}x\right)^2.\]
Here, $D^s u$ denotes\footnote{For $s<0$, we define $\widehat{D^su}(0)=0$.} the action on $u$ of the Fourier multiplier with symbol $|n|^s$:
\[D^s u:= |\nabla|^s u =\mathcal{F}^{-1}(|n|^s\widehat{u}).\]

The measure 
\[L_N\otimes \vec{\nu}_{s,N}^\perp(\mathrm{d}\vec{u})\]
appearing in \eqref{eqn: mu-COV} is the product of Lebesgue measure $L_N$ on
\[\mathcal{E}_N\times \mathcal{E}_N:=(\pi_N L^2(\mathbb{T}^3))^2\]
and the projection $\vec{\nu}_{s,N}^\perp=(\mathrm{id}-\pi_N)_*\vec{\nu}_{s,N}$ of $\vec{\nu}_{s,N}$ on
\[\mathcal{E}_N^\perp\times \mathcal{E}_N^\perp:=((\mathrm{id}-\pi_N)L^2(\mathbb{T}^3))^2.\]
$Z_{s,N}$ is a normalization factor.

Differentiating \eqref{eqn: mu-COV} and using the invariance of the Hamiltonian, we obtain
\[-Z_{s,N}^{-1}\int_{A}\frac{1}{2}\frac{\mathrm{d}}{\mathrm{d}t}\|\Phi_N(t)(\pi_N u)\|^2 e^{-\frac{1}{2}\|\Phi_N(t)(\pi_N u)\|^2}\mathrm{d} L_N\otimes \vec{\nu}_{s,N}^\perp(\mathrm{d}u).\]
Denote  by
\[u_N:= \pi_N u, \quad v_N:=\pi_N v\]
the projections of the solution on Fourier frequencies less than or equal to $N$. The derivative of the energy is
\begin{align}
  \frac{\mathrm{d}}{\mathrm{d}t}\|\Phi_N(t)(\pi_N u)\|^2&=\partial_t \left(\frac{1}{2}\int_{\mathbb{T}^3} (D^s v_N)^2+ (D^{s+1} u_N)^2+\frac{1}{2}\big(\int_{\mathbb{T}^3}u_N\big)^2\right)\\
&+\partial_t\left( E(u_N,v_N)-\frac{1}{k+1}\int_{\mathbb{T}^3} u_N^{k+1} \right)\nonumber\\
  &= \int_{\mathbb{T}^3} D^{2s}v_N (-u_N^k)+\int_{\mathbb{T}^3} u_N\int_{\mathbb{T}^3}v_N-\frac{\mathrm{d}}{\mathrm{d}t}\big(\frac{1}{k+1}\int_{\mathbb{T}^3} u_N^{k+1}\big). \label{derivative-energy}
\end{align}

We now rewrite the first quantity in \eqref{derivative-energy} as
\begin{align}
  \int u_N^k D^{2s}v_N =&~ k\int D^s v_N D^s u_N u_N^{k-1} \nonumber \\
                       &+\int D^s v_N (D^su_N^k-k u_N^{k-1} D^s u_N) \nonumber \\
                       =&~ \frac{k}{2} \partial_t \int (D^s u_N)^2 u_N^{k-1} - \frac{k(k-1)}{2}\int (D^s u_N)^2v_N u_N^{k-2} \label{eqn: pre}\\
                       &+\int D^s v_N (D^su_N^k-k u_N^{k-1} D^s u_N). \label{eqn: commutator}
\end{align}
The quantity $(D^s u_N)^2$ is divergent on the support of $\vec{\nu}_s$, and so requires a renormalization. This was one of the innovations in \cite{OT}. Following the notation introduced there, we define
\begin{equation}\label{eqn: Q-def}
  Q_{s,N}(f):=(D^s f)^2-\sigma_N,
  \end{equation}
with
\[\sigma_N := \mathbb{E}_{\vec{\nu}_s}[(D^s \pi_N u)^2]\sim N.\]

Defining the \emph{renormalized energy} by
\begin{equation}\label{eqn: R-energy}
  \mathscr{E}_{s,N}(u,v):= \frac{1}{2}\big(\int_{\mathbb{T}^3} (D^s v_N)^2+\int_{\mathbb{T}^3} (D^{s+1} u_N)^2+\big(\int_{\mathbb{T}^3} u_N\big)^2\big) -\frac{k}{2}\int_{\mathbb{T}^3} Q_{s,N}(u_N)u_N^{k-1},
\end{equation}
the result of the above computations is the following expression for the time derivative of $\mathscr{E}_{s,N}(u,v)$:
\begin{equation}
  \begin{split}\label{eqn: dt-energy}
    \partial_t \mathscr{E}_{s,N}(\pi_N\Phi_N(t)(u,v))=& -\frac{k(k-1)}{2}\int_{\mathbb{T}^2} Q_{s,N}(u_N)v_N u_N^{k-2}\\
    &+\int_{\mathbb{T}^2} D^s v_N (D^su_N^k-k u_N^{k-1} D^s u_N)+\int_{\mathbb{T}^3} u_Nv_N.
  \end{split}
\end{equation}
The idea in \cite{T1} is to pass from $\vec{\nu}_s$ to the measure
\[e^{-R_s(u,v)}\mathrm{d}\vec{\nu}_s,\]
where $R_s(u,v)$ is a limit of the terms
\begin{equation}\label{eqn: RskN}
  R_{k,s,N}:=\frac{k}{2}\int_{\mathbb{T}^3} Q_{s,N}(u_N)u_N^{k-1}+\frac{1}{k+1}\int_{\mathbb{T}^3} u_N^{k+1}
\end{equation}
appearing in the renormalized energy \eqref{eqn: R-energy}. We must then estimate the time derivative \eqref{eqn: dt-energy}.

The following two propositions contain the main technical results needed to close the argument.

\begin{proposition}\label{lem: Lp-bound}
  Let $s> \frac{5}{2}$. Define
\[E_N(u,v):=\frac{1}{2}\int |\nabla u_N|^2+\frac{1}{2}\int v_N^2+\frac{1}{k+1}\int u_N^{k+1} .\]
Then for $p<\infty$ there exists $q>0$ and $C_p>0$ such that 
\begin{equation*}
\sup\limits_{N \in\mathbb{N}} \norm{e^{-R_{k,s,N}(u)-E^q_N(u,v)}}_{L^p(\vec{\nu}_s)} \leq C_p<\infty\label{eqn: Lp-bound}
\end{equation*}
and moreover,
\begin{equation}\label{eqn: Lp-conv}
\lim\limits_{N\rightarrow\infty} e^{-R_{k,s,N}(u)-E^q_N(u,v)}= e^{F_{s,k}(u,v)}\quad in\,\, L^p(\vec{\nu}_s). 
\end{equation}
In particular, for any $\sigma<s-\frac{1}{2}$, the restrictions to $\vec{H}^\sigma$ of the measures
\[\mathrm{d}\vec{\rho}_{s,N}=\mathcal{Z}_{s,N}^{-1}e^{-R_{k,s,N}(\vec{u})-E^q}\,\vec{\nu}_s(\mathrm{d}\vec{u})\]
converge in total variation to a measure $\vec{\rho}_s$.
\end{proposition}

\begin{proposition}\label{theorem:energy_estimate}
	Let $s>\frac{5}{2}$. Given $R_0>0$ there exists $C=C(R_0)$ such that, for all $p\ge 1$ finite, we have
	\begin{equation}\label{eqn: energy-estimate}
		\mathbb{E}_{\vec{\rho}_s}[\mathbf{1}_{B_{R_0}(\vec{u})}\left|\partial_t \mathscr{E}_{s,N}(\pi_N\Phi_N(t)(\vec{u}))|_{t=0}\right|^p]^{\frac{1}{p}}\leq Cp
	\end{equation}
	where $C$ can be taken independent of $N$.
\end{proposition}
As in \cite{GOTW}, we obtain the estimates necessary to the construction of our measure from a variational bound introduced in \cite{BG}.

\paragraph{Finishing the proof of Theorem \ref{thm: quasi-invariance}.}We now indicate how to finish the proof given Propositions \ref{lem: Lp-bound} and \ref{theorem:energy_estimate}. Since this part of the argument requires no modification from \cite{GOTW}, we refer the reader to that paper for details.

Let $\sigma<s-\frac{3}{2}$, $d=3$ and $k=3$ or $k=5$. Fix $R>0$ and let $A\subset B_R\subset \vec{H}^\sigma$ be a Borel measurable set (for the topology induced by the norm) such that
\[\vec{\rho}_{s,N}(A)=0.\] 
When $k=3$, a simple application of Gronwall's lemma and conservation of the energy\footnote{See \cite[Lemma 2.5]{GOTW}.} shows that for any $T>0$, there is a radius $C(T,R)$ such that for $|t|\le T$,
\begin{equation}\label{eqn: bound-2}
\Phi_N(t)(B_R)\subset B_{C(T,R)}.
\end{equation}
The estimate \eqref{eqn: bound-2} in case $k=5$ is proved in Appendix \ref{sec: critical}.

By the change of variables formula, we have
\[\vec{\rho}_{s,N}(\Phi_N(t)(A))=\mathcal{Z}_{s,N}^{-1}\int_{A}e^{-E_N(\pi_N\Phi_N(t)(u,v))-E^q_N(u,v)}\mathrm{d} L_N\otimes \vec{\nu}_{s,N}^\perp(\mathrm{d}u).\]
Differentiating and using \eqref{eqn: energy-estimate}, we find
\[\frac{\mathrm{d}}{\mathrm{d}t}\vec{\rho}_{s,N}(\Phi_N(t)(A))\le C_R p\cdot  \vec{\rho}_{s,N}(\Phi_N(t)(A))^{1-\frac{1}{p}}\]
This inequality is equivalent to
\begin{equation}\label{eqn: deriv-ineq}
  \frac{\mathrm{d}}{\mathrm{d}t}\vec{\rho}_s(\Phi_N(t)(A))^{\frac{1}{p}}\le C_R.\end{equation}
The linear dependence on $p$ on the right side is essential in \eqref{eqn: energy-estimate} plays an essential role here. Integrating \eqref{eqn: deriv-ineq}, we find
\begin{align*}
  \vec{\rho}_{s,N}(\Phi_N(t)(A))&\le (\vec{\rho}_{s,N}(A)^{1/p}+C_Rt)^p\\
&\le 2^p \vec{\rho}_{s,N}(A) + C_R^p 2^pt^p.
\end{align*}
Taking $t\le \frac{1}{4C_R}$ and $p$ large enough, we find that
\begin{equation}\label{eqn: eps-bound}
  \vec{\rho}_{s,N}(\Phi_N(t)(A))<\varepsilon,
  \end{equation}
uniformly in $N$, for any $A\subset B_R\subset \vec{H}^\sigma(\mathbb{T}^3)$.
Theorem \ref{thm: quasi-invariance} follows from \eqref{eqn: eps-bound} by a soft approximation argument identical to that in \cite[Section 5.2]{OT}.

For Theorem \ref{thm: bound}, the estimate \eqref{eqn: bound-2} is replaced by  \eqref{eqn: det-bound}, where $T$ now depends on $R$, but otherwise the proof proceeds as before.

\section{Basic estimates}
We use the notation $A\lesssim B$ to mean that $A\le CB$ where $C$ is an unspecified constant independent of $N$ whose exact value is unimportant for the argument.
\subsection{Littlewood-Paley theory} We denote the Fourier transform of a function $u\in L^1(\mathbb{T}^3)$ by
\[\widehat{u}(n)=\frac{1}{(2\pi)^3}\int u(x)e^{in\cdot x}\,\mathrm{d}x.\]
The Fourier transform of $u\in \mathcal{D}'(\mathbb{T}^3)$, the space of distributions on $\mathbb{T}^3$ is defined in the usual fashion by duality. As we have already stated above, we denote by $\pi_N$ the Dirichlet truncation of a distribution in Fourier space:
\begin{equation}\label{eqn: dirichlet-proj}
  \pi_N u =\sum_{|n|\le N} \widehat{u}(n)e^{in\cdot x}.
  \end{equation}

We make extensive use of Littlewood-Paley theory in its dyadic decomposition incarnation. See \cite{BDC} for a thorough treatment. Following these authors, we $B(\xi,r)$ denote the ball in $\mathbb{R}^3$ of radius $r$ around a point $\xi$ in phase space.

Consider functions $\chi$, $\tilde{\chi}$ such that
\begin{align*}
\mathrm{supp}\chi&\subset B\big(0,\frac{4}{3}\big),\\
\mathrm{supp}\tilde{\chi}&\subset B\big(0,\frac{3}{4}\big)\setminus B\big(0,\frac{3}{8}\big).
\end{align*}
We define $\psi_{-1}=\chi$ and
\[\psi_j(\cdot)=\tilde{\chi}(2^{-j}\cdot), \quad j\ge 0.\]
We also define $\mathbf{P}_j$, the \textit{Littlewood-Paley projector}, associated to symbol $\psi_j$ by
\[\mathbf{P}_j u(x)=\big(\psi_j(\nabla)u\big)(x)=\sum_{n\in \mathbb{Z}^3} \tilde{\chi}_j(n)\widehat{u}(n)e^{in\cdot x}.\]
 We define the (low-high) paraproduct $T_ab$ by
\[T_ab:= \sum_{j=-1}^{\infty}S_{j-1}a\mathbf{P}_jb,\] where 
\[S_j := \sum_{k\leq j-1}\mathbf{P}_k.\]

\subsection{Basic estimates for Besov-type norms}
We recall the definition of Besov spaces through the Littlewood-Paley decomposition. For $s\in \mathbb{R}$ and $1\le p,q\le \infty$, the Besov space $B^s_{p,q}(\mathbb{T}^3)$ is the set of distributions $u$ on $\mathbb{T}^3$ such that
\[\|u\|_{B^s_{p,q}(\mathbb{T}^3)}=\left\|\big(2^{sj}\|\mathbf{P}_j u\|_{L_x^p}\big)_{j\ge 0}\right\|_{\ell^q_j}<\infty.\]
For $p=q=2$, this corresponds to the Sobolev spaces $H^s(\mathbb{T}^3)$, while for $s>0$, $s\notin \mathbb{Z}$, this is the H\"older space $C^s(\mathbb{T}^3)$. We define H\"older spaces for negative $s$ or for $s\in \Z$ by setting $C^s: = B_{\infty,\infty}^{s}$. We again refer to \cite{BDC,MWX} for details.
Following \cite{GOTW}, we denote
\[\vec{B}^s_{p,q}(\mathbb{T}^3)=B^s_{p,q}(\mathbb{T}^3)\times B^{s-1}_{p,q}(\mathbb{T}^3).\]

\begin{lemma}
The following estimates hold
\begin{enumerate}[(i)]

\item Let $s_0,s_1,s\in\R$ and $\nu\in[0,1]$ be such that $s=(1-\nu)s_0+\nu s_1$. Then,
\begin{equation}\label{EQU: Interpolation}
  \norm{u}_{H^s}\lesssim \norm{u}_{H^{s_0}}^{1-\nu} \norm{u}_{H^{s_1}}^{\nu}.
\end{equation}
\item Let $s_0,s_1\in\R$ and $p_0,p_1,q_0,q_1\in[1,\infty]$. Then,
\begin{align}\label{EQU: Intermediate embeddings}
    \begin{split}
        \norm{u}_{B_{p_0,q_0}^{s_0}}&\lesssim \norm{u}_{B_{p_1,q_1}^{s_1}}\quad \textnormal{for } s_0\leq s_1, p_0\leq p_1 \textnormal{ and } q_0\geq q_1, \\
        \norm{u}_{B_{p_0,q_0}^{s_0}} &\lesssim \norm{u}_{B_{p_0,\infty}^{s_1}} \quad \textnormal{for } s_0 < s_1,\\
        \norm{u}_{B_{p_0,\infty}^0}&\lesssim \norm{u}_{L^p}\lesssim \norm{u}_{B_{p_0,1}^0} .
    \end{split}
\end{align}
\item Let $s>0$. Then,
\begin{equation}\label{EQU: Algebra property}
    \norm{uv}_{C^{s}}\lesssim \norm{u}_{C^{s}}\norm{v}_{C^{s}}.
\end{equation}
\item Let $1\leq p_1\leq p_0\leq \infty$, $q\in[1,\infty]$ and $s_1=s_0+d\left(\frac{1}{p_1}-\frac{1}{p_0}\right)$. Then,
\begin{equation}\label{EQU: Besov embedding}
    \norm{u}_{B^{s_0}_{p_0,q}}\lesssim \norm{u}_{B^{s_1}_{p_1,q}}.
\end{equation}
\item Let $s\in\R$ and $p,p',q,q'\in[1,\infty]$ be such that $\frac{1}{p}+\frac{1}{p'}=1$ and $\frac{1}{q}+\frac{1}{q'}=1$. Then,
\begin{equation}\label{EQU: Duality}
    \left|\int_{\T^d} uv\,dx\right| \leq\norm{u}_{B^{s}_{p,q}}\norm{u}_{B^{-s}_{p',q'}}.
\end{equation}
\item Let $p,p_0,p_1,p_2,p_3\in[1,\infty]$ be such that $\frac{1}{p_0}+\frac{1}{p_1}=1$ and $\frac{1}{p_2}+\frac{1}{p_3}=1$.. Then for $s>0$,
\begin{equation}\label{EQU: Fractional Leibniz rule}
    \norm{uv}_{B^{s}_{p,q}}\lesssim \norm{u}_{B^{s}_{p_0,q}}\norm{v}_{L^{p_1}}+\norm{v}_{B^{s}_{p_2,q}}\norm{u}_{L^{p_3}}.
\end{equation}
\item Let $m>0$ be an integer, $s>0$ and $q,p,p_0,p_1\in[0,\infty]$ satisfy $\frac{1}{p_0}+\frac{1}{p_1}=\frac{1}{p}$. Then 
\begin{equation}\label{EQU: Fracational Leibniz rule cor}
    \norm{u^{m+1}}_{B^{s}_{p,q}}\lesssim \norm{u^m}_{L^{p_0}}\norm{u}_{B^s_{p_1,q}}.
\end{equation}

\item Let $s_0<0<s_1$ be such that $s_0+s_1>0$. Then,
\begin{equation}\label{EQU: Multiplicative estimate}
    \norm{uv}_{C^{s_0}}\lesssim \norm{u}_{C^{s_0}}\norm{v}_{C^{s_1}}.
\end{equation}
\end{enumerate}
\end{lemma}

We refer to \cite{BDC, MWX} for proofs. We also note the Bernstein inequality:
\begin{equation}\label{eqn: bernstein}
\|\mathbf{P}_j u\|_{L^p}\lesssim  2^{3j(\frac{1}{p}-\frac{1}{q})}\|\mathbf{P}_j u\|_{L^q}
\end{equation}
for $p\le q\le \infty$.

\subsection{Wiener chaos and hypercontractivity}
The hypercontractive estimate is a key tool to estimate nonlinear functions of Gaussian random variables. We recall this estimate here. See S. Janson's book \cite{janson} for more information on hypercontractivity and Wiener chaos spaces.

Let $X_n$, $n\ge 1$ be a sequence of i.i.d. standard Gaussian random variables. We define $\mathcal{H}_k$, the homogeneous Wiener chaos of order $k$ to be the closed span of the polynomials
\begin{equation}\label{eqn: hermite-product}
\prod_{n=1}^\infty H_{k_n}(X_n),
\end{equation}
where $H_j$ is the Hermite polynomial of degree $j$  and $k=\sum_{n=1}^\infty k_n$. Note that since $H_0(x)=1$, the formally infinite product in \eqref{eqn: } has only finitely many non-trivial factors. We have
\[L^2(\Omega,\mathcal{F},\mathbb{P})=\oplus_{k=0}^\infty \mathcal{H}_k,\]
where $\mathcal{F}$ is the $\sigma$-algebra generated by the Gaussian random variables $X_n$, $n\ge 1$.

The next lemma gives the crucial hypercontractivity estimate. See \cite{janson} for a proof.
\begin{proposition} Let $p\ge 2$ be finite. If $X\in \mathcal{H}_k$, then
  \begin{equation}\label{eqn: wiener-chaos}
 \big( E[|X|^p]\big)^{\frac{1}{p}}\le (p-1)^{\frac{k}{2}}\big(E[|X|^2]\big)^{\frac{1}{2}}.
  \end{equation}
\end{proposition}

\section{Energy estimate for fractional $s$}
In this section, we derive Proposition \ref{theorem:energy_estimate}, the energy estimate  for fractional $s$. This is the analogue of \cite[estimate 3]{GOTW} to general nonlinearities and fractional regularity $s>5/2$. 

The possibility of fractional $s$ makes our derivation more involved, since we cannot integrate by parts as in \cite[Equation 1.25]{OT} and \cite[Equation 3.5]{GOTW} to remove the most singular spatial derivatives in the time derivative of the energy. Instead, we perform a higher order expansion to exploit the cancellation in the commutator term $ku_N^{k-1}D^su_N-D^su_N^k$ appearing in \eqref{eqn:energy_derivative}.

Recall that
\begin{equation}\label{eqn:energy_derivative}
	\begin{split}
		\partial_t \mathscr{E}_{s,N}(\pi_N\Phi_N(t)(\vec{u}))|_{t=0} &= \frac{k(k-1)}{2}\int Q_{s,N}(u_N)v_Nu_N^{k-2} \\
		&+ \int D^sv_N(ku_N^{k-1}D^su_N-D^su_N^k) + \int u_N \int v_N.
	\end{split}
\end{equation}
\begin{proposition}\label{theorem:energy_derivative}
	For $s>\frac{5}{2}$, there exists $\sigma < s-\frac{1}{2}$, such that, for $\varepsilon$ sufficiently small,
	\begin{equation}\label{ineq:energy_derivative}
	\begin{split}
		&\left|\partial_t \mathscr{E}_{s,N}(\pi_N\Phi_N(t)(\vec{u}))|_{t=0} \right|\\
\lesssim ~&(1+\norm{\vec{u_N}}_{\vec{H}^{\sigma}}^{k-1})(1+\norm{Q_{s,N}(u_N)}_{C^{-1-\varepsilon}}
		+\norm{u_N}_{C^{s-\frac{1}{2}-\varepsilon}}\norm{v_N}_{C^{s-\frac{3}{2}-\varepsilon}}\\
		&+\norm{D^{s-\frac{3}{2}}vD^{s+\frac{1}{2}}u}_{C^{-1-\varepsilon}}+\norm{D^{s-\frac{3}{2}}vD^{s-\frac{1}{2}}u}_{C^{-\varepsilon}}). 
		\end{split}
	\end{equation}
	The implicit constants are uniform in $N$.
\end{proposition}
Then Proposition \ref{theorem:energy_estimate} follows from (\ref{ineq:energy_derivative}) and Lemma \ref{lemma:chaos_estimate}.

Our goal is to prove uniform in $N$ estimates for the derivative of the energy \eqref{eqn:energy_derivative}. For this reason and for simplicity of notation, in this section we omit the subscripts $N$ on $u$ and $v$ in deriving the estimates.

\begin{lemma}\label{lemma:chaos_estimate}
	We have
	\begin{align}
	\norm{v}_{L^p(dv_s)C^{s-\frac{3}{2}-\varepsilon}(\mathbb{T}^3)}&\lesssim \sqrt{p}, \label{eqn: v-C-estimate}\\
	\norm{u}_{L^p(\mathrm{d}\vec{\nu}_s)C^{s-\frac{1}{2}-\varepsilon}(\mathbb{T}^3)}&\lesssim \sqrt{p}. \label{eqn: u-C-estimate}
	\end{align}
	If $\alpha+\beta>\frac{3}{2}$, $\alpha, \beta \geq 0$, then for $\gamma < \min\{\alpha-\frac{3}{2},\beta-\frac{3}{2},\alpha+\beta-3\}$,
	\begin{equation}\label{eqn: third}
	\norm{D^{s-\alpha}vD^{s+1-\beta}u}_{L^p(\mathrm{d}\vec{\nu}_s)C^{\gamma}(\mathbb{T}^3)}\lesssim p.
	\end{equation}
\end{lemma}
\begin{proof}
	We only prove the third bound, \eqref{eqn: third}. The first two can be proved in a similar manner. To simplify the notations, let $L^p_w, L^q_x$ denote $L^p(\mathrm{d}\vec{\nu}_s), L^q(\mathbb{T}^3)$ respectively. By Bernstein's inequality \eqref{eqn: bernstein}, Fubini's theorem and Wiener chaos estimate \eqref{eqn: wiener-chaos},
	\begin{equation}\label{chaos estimate_bound}
	\begin{split}
	\norm{D^{s-\alpha}v D^{s+1-\beta}u}_{L^p_{\omega}C_x^{\gamma}}
	&\leq \norm{D^{s-\alpha}v D^{s+1-\beta}u}_{L^p_{\omega} B^{\gamma+\frac{3}{p}}_{p,1}} \\
	&= \norm{\sum_j 2^{j(\gamma+\frac{3}{p})}\norm{\mathbf{P}_j(D^{s-\alpha}v D^{s+1-\beta}u)}_{L^p_x}}_{L^p_{\omega}}\\
	&\leq \sum_j 2^{j(\gamma+\frac{3}{p})}\norm{\mathbf{P}_j(D^{s-\alpha}v D^{s+1-\beta}u)}_{L^p_xL^p_{\omega}}\\
	&\lesssim p\sum_j 2^{j(\gamma+\frac{3}{p})}\norm{\mathbf{P}_j(D^{s-\alpha}v D^{s+1-\beta}u)}_{L^p_xL^2_{\omega}}.
	\end{split}
	\end{equation}
	On the other hand, by series representation (\ref{series}),
	\begin{equation}\label{ineq: expectation of dyadic}
	\begin{split}
	\mathbb{E}_{\omega}[|\mathbf{P}_j(D^{s-\alpha}v D^{s+1-\beta}u)|^2]
	&= \sum_{n\sim 2^j}\sum_{n_1}\dfrac{1}{\jbrac{n_1}^{2\alpha}\jbrac{n-n_1}^{2\beta}}
	\end{split}
	\end{equation}
	 We claim the following convolution estimate: for any $\varepsilon>0$,
	\begin{equation}\label{ineq: convolution}
		\sum_{n_1\in\Z^3}\dfrac{1}{\jbrac{n_1}^{2\alpha}\jbrac{n-n_1}^{2\beta}} \lesssim \jbrac{n}^{-\min\{2\alpha,2\beta,2\alpha+2\beta-3\}+\varepsilon}
	\end{equation}
	provided $\alpha+\beta>\frac{3}{2}$ and $\alpha, \beta\geq 0$. Hence (\ref{chaos estimate_bound}) is bounded by $p$ if $p$ is large and $\gamma<\min\{\alpha-\frac{3}{2},\beta-\frac{3}{2},\alpha+\beta-3\}$.
	
	To prove \eqref{ineq: convolution}, we follow the idea of \cite[Lemma 4.1]{MWX}. Split the index set of the summation into
	\begin{align*}
		\mathcal{A}_1 &:= \{n_1\in\Z^3 :\; |n_1|>2|n| \},\\
		\mathcal{A}_2 &:= \{n_1\in\Z^3 :\; |n_1|<\frac{1}{2}|n\},\\
		\mathcal{A}_3 &:= \{n_1\in\Z^3 :\; |n-n_1|<\frac{1}{2}|n|\},\\
		\mathcal{A}_4 &:= \Z^3\backslash\bigcup_{j=1}^3\mathcal{A}_j,
	\end{align*}
	and bound each parts separately. If $n_1\in\mathcal{A}_1$, then $|n-n_1|\geq |n_1|-|n| > \frac{1}{2}|n_1|$. Using $2\alpha+2\beta>3$, we have
	\begin{align*}
		\sum_{n_1\in\mathcal{A}_1}\dfrac{1}{\jbrac{n_1}^{2\alpha}\jbrac{n-n_1}^{2\beta}} \lesssim \sum_{n_1\in\Z^3} \dfrac{1}{\jbrac{n_1}^{2\alpha+2\beta}}\lesssim \jbrac{n}^{-(2\alpha+2\beta-3)}.
	\end{align*}
	For $n_1\in\mathcal{A}_2$, we have $|n-n_1|\geq |n|-|n_1|>\frac{1}{2}|n|$, then
	\begin{align*}
		\sum_{n_1\in\mathcal{A}_1}\dfrac{1}{\jbrac{n_1}^{2\alpha}\jbrac{n-n_1}^{2\beta}} \lesssim \dfrac{1}{\jbrac{n}^{2\beta}}\sum_{|n_1|<\frac{1}{2}|n|}\dfrac{1}{\jbrac{n_1}^{2\alpha}}\lesssim \jbrac{n}^{-2\beta -\min\{-\varepsilon, 2\alpha-3\}}.
	\end{align*}
	Here, $-\varepsilon$ is used to bound the logarithm factor when $2\alpha=3$. Similar bound holds for $\mathcal{A}_3$ by symmetry. Finally, if $n_1\in \mathcal{A}_4$, then $|n-n_1|\geq \frac{1}{2}|n|$ and $\frac{1}{2}|n|\leq |n_1| \leq 2|n|$. Therefore,
	\begin{align*}
		\sum_{n_1\in\mathcal{A}_4}\dfrac{1}{\jbrac{n_1}^{2\alpha}\jbrac{n-n_1}^{2\beta}} \lesssim \dfrac{1}{\jbrac{n}^{2\beta}}\sum_{\frac{1}{2}|n|\leq |n_1| \leq 2|n|}\dfrac{1}{\jbrac{n_1}^{2\alpha}} \lesssim \jbrac{n}^{-(2\alpha+2\beta-3)}
	\end{align*}
\end{proof}

To prove Proposition \ref{theorem:energy_derivative}, we only need to bound the second term on the RHS of (\ref{eqn:energy_derivative}). The other terms are the same as those in \cite[Proposition 5.1]{GOTW}. 

\begin{lemma}\label{high term lemma}
For any $\varepsilon>0$, we have	$\norm{kT_{u^{k-1}}u-u^k}_{B^{\alpha+\beta-\varepsilon}_{1,1}}\lesssim \norm{u}_{B^{\alpha}_{p,\infty}}\norm{u}_{B^{\beta}_{q,\infty}}\norm{u}_{L^{\infty}_x}^{k-2}$, provided $\alpha+\beta > 0$, and $\frac{1}{p}+\frac{1}{q}=1$.
\end{lemma}
\begin{proof}
	A more general $\mathbb{R}^d$ version can be found in \cite[Theorem 2.92]{BDC}. For our case, note
	\begin{equation}
	\begin{split}
	u^k &= \sum_{j=-1}^{\infty} (S_{j+1}u)^k-(S_ju)^k\\
	&= \sum_{j=-1}^{\infty} \mathbf{P}_ju\sum_{\ell=0}^{k-1}(S_{j+1}u)^\ell(S_ju)^{k-\ell-1}.
	\end{split}
	\end{equation}
	It suffices to bound $T_{u^{k-1}}u-\sum_j \mathbf{P}_ju(S_{j+1}u)^{k-1}$, the other terms are similar. Since $\alpha+\beta>0$, we may assume $\beta>0$. Taking the $L^1_x$ norm, we have
\begin{align*}
	&\norm{\mathbf{P}_n\left(T_{u^{k-1}}u-\sum_j \mathbf{P}_ju(S_{j+1}u)^{k-1}\right)}_{L_x^1}\\
	=& \norm{\mathbf{P}_n\sum_{j\geq n-3}\mathbf{P}_ju\left(S_{j-1}u^{k-1}-(S_{j+1}u)^{k-1}\right)}_{L^1_x}\\
	\lesssim& \sum_{j\geq n-3}\norm{\mathbf{P}_ju}_{L^p_x}\norm{S_{j-1}(u^{k-1}-(S_{j-k}u)^{k-1})+\left((S_{j-k}u)^{k-1}-(S_{j+1}u)^{k-1}\right)}_{L^{q}_x}\\
	\lesssim& \sum_{j\geq n-3}\norm{\mathbf{P}_ju}_{L_x^p}\left(\norm{u^{k-1}-(S_{j-k}u)^{k-1}}_{L^q_x} + \norm{(S_{j-k}u)^{k-1}-(S_{j+1}u)^{k-1}}_{L^{q}_x}\right).
\end{align*}
Expressing $u^{k-1}-(S_{j-1}u)^{k-1}$ in terms of products of lower-order quantities, we find that the previous expression is bounded up to a constant factor by
\begin{align*}
& \sum_{j\geq n-3}\norm{\mathbf{P}_ju}_{L_x^p}\sum_{m\geq j-k}\norm{\mathbf{P}_mu}_{L^q_x}\sum_{\ell =0}^{k-2}\norm{S_{j-k}u}^{\ell}_{L^{\infty}_x}\left(\norm{u}_{L_x^{\infty}}^{k-\ell-2}+\norm{S_{j+1}u}_{L^{\infty}_x}^{k-\ell-2}\right)\\
	\lesssim& \sum_{j\geq n-3} 2^{-j\alpha}\norm{u}_{B^{\alpha}_{p,\infty}}\sum_{m\geq j-k}2^{-m\beta}\norm{u}_{B_{q,\infty}^{\beta}}\norm{u}_{L^{\infty}_x}^{k-2}\\
	\lesssim& 2^{-n(\alpha+\beta)}\norm{u}_{B^{\alpha}_{p,\infty}}\norm{u}_{B_{q,\infty}^{\beta}}\norm{u}_{L^{\infty}_x}^{k-2}.
	\end{align*}
\end{proof}

The next lemma allows us to replace $u^{k-1}D^su$ by $T_{u^{k-1}}D^su$.
\begin{lemma}\label{low and resonant term lemma}
	$\norm{T_ab-ab}_{B_{1,1}^{\alpha+\beta}} \lesssim \norm{a}_{B_{1,\infty}^{\alpha}}\norm{b}_{C^{\beta}}$, provided $\beta < 0$ and $\alpha+\beta>0$.
\end{lemma}
\begin{proof}
	The proof is a straightforward application of the definition of the paraproduct and Besov spaces.
\end{proof}

The remaining difference $D^sT_{u^{k-1}}u-T_{u^{k-1}}D^su$ cannot be bounded directly. The following decomposition is the main result describing the regularity of this commutator:
\begin{lemma}\label{lemma:decomposition_lemma} Given $s>0$,
	$D^s(T_w u) - T_w(D^s u) = F_1 + F_2 + R$, where
	\begin{align*}
	F_1 =&s\sum_{j = 1}^3T_{\partial_j w}(\partial_j D^{s-2} u),\\
	F_2 =&\frac{s(s-2)}{2}\sum_{j=1}^3T_{\partial_j^2 w}(\partial_j^2 D^{s-4}u)+\frac{s}{2}T_{D^2w}D^{s-2}u.
	\end{align*}
	And for any $\beta \in \R$, $\rho < 3$, 
	\begin{align}
		\norm{R}_{B_{1,1}^{\beta + \rho}} \lesssim \norm{w}_{B_{1,\infty}^{{\rho}}}\norm{u}_{C^{s+\beta}}.
	\end{align}
\end{lemma}

\begin{remark}
	If we only expand once, then the reminder has the same bound but with more restricted $\rho$. More precisely, with the notations above, for $\rho <2$,
	\begin{align}\label{ineq: second expansion}
		\norm{F_2+R}_{B_{1,1}^{\beta + \rho}} \lesssim \norm{w}_{B_{1,\infty}^{{\rho}}}\norm{u}_{C^{s+\beta}},
	\end{align}
	and for $\rho<1$,
\end{remark}
\begin{proof}
	Let $m\in C^{\infty}(\mathbb{R})$ be a bump function such that $m=1$ on $[-\frac{1}{2},\frac{1}{2}]$ and is supported on $[-\frac{1}{2}-,\frac{1}{2}+]$. Then,
	\begin{align*}
	&(D^s(T_w u) - T_w(D^s u))(x)\\
	 =&\sum_{n_1,n_2\in\mathbb{Z}^3} (|n_1+n_2|^s-|n_2|^s)\widehat{w}(n_1)\widehat{u}(n_2)m(\frac{|n_1|}{|n_2|})e^{i (n_1+n_2)\cdot x}.
	\end{align*}
	By Taylor's theorem, for $n_2\neq 0$\footnote{If $n_2=0$, then on the support of $m(\frac{|n_1|}{|n_2|})$, we have $n_1=0$. That why we don't need to estimate $\mathbf{P}_{-1}R$.},
	\begin{align*}
	&|n_1+n_2|^s-|n_2|^s\\
	=&~s|n_2|^{s-1}n_2\cdot n_1 +\frac{1}{2}\left(s(s-2)|n_2|^{s-4}(n_2\cdot n_1)^2 + s|n_2|^{s-2}|n_1|^2\right)+ R_1
	\end{align*}
	where the first 2 terms correspond to $F_1$ and $F_2$, and
	\begin{align*}
	R_1(n_1,n_2):=\dfrac{s(s-2)}{2} \int_0^1 & (1-t)^2\Big( 3|n_2+tn_1|^{s-4}(n_2+tn_1)\cdot n_1 |n_1|^2\\
	&+ (s-4)|n_2+tn_1|^{s-6}((n_2+tn_1)\cdot n_1)^3\Big)\mathrm{d}t.
	\end{align*}
	
Define $R$ by
\[R= \sum_{n_1,n_2\in \mathbb{Z}^3} R_1(n_1,n_2)\widehat{u}(n_1)\widehat{w}(n_2)e^{i(n_1+n_2)\cdot x}.\]
We now estimate $R$. First, we can write $R_1 =\sum_{|\alpha|=3}C_{\alpha}(n_1,n_2)n_1^{\alpha}$, where $\alpha$ is a 3d multi-index and $C_{\alpha}$ can be extended by homogeneity to a function on $\mathbb{R}^6$ such that 
\[C_{\alpha}(\lambda\xi_1,\lambda\xi_2) = \lambda^{s-3}C_{\alpha}(\xi_1,\xi_2)\] for any $\lambda>0$, and is smooth on the support of $m(\frac{|\xi_1|}{|\xi_2|})\psi_j(\xi_1+\xi_2)$ for any $j\ge 0$.
	
	To bound $\norm{\mathbf{P}_jR}_{L^1_x}$, set 
	\[K_{\alpha,j}(\xi_1,\xi_2) := \psi_j(\xi_1+\xi_2)m(\frac{|\xi_1|}{|\xi_2|})C_{\alpha}(\xi_1,\xi_2),\] 
	and define $h_j\in L^2(\mathbb{T}^3\times\mathbb{T}^3)$ and $H_j\in L^2(\mathbb{R}^3\times\mathbb{R}^3)$ by
	\begin{align*}
	h_{\alpha,j}(y,z):=&\sum_{|\alpha|=3}(\mathcal{F}_{\mathbb{T}^6}^{-1}K_{\alpha,j})(y,z),\\
	H_{\alpha,j}(y,z):=&\sum_{|\alpha|=3}(\mathcal{F}_{\mathbb{R}^6}^{-1}K_{\alpha,j})(y,z).
	\end{align*}
	By the Poisson summation formula, 
	\[h_{\alpha,j}(y,z) = \sum_{(m_1, m_2)\in \mathbb{Z}^6}H_{\alpha,j}(y+m_1,z+m_2).\] 
	Hence,
	\begin{align*}
	\norm{h_{\alpha,j}}_{L^1(\mathbb{T}^6)}&\leq \norm{H_{\alpha,j}}_{L^1(\mathbb{R}^6)}\\
	 &= \norm{2^{j(s-3)}\cdot2^{6j}H_{\alpha,0}(2^j\cdot, 2^j\cdot)}_{L^1(\mathbb{R}^6)}\\
	 &= 2^{j(s-3)}\norm{H_{\alpha,0}}_{L^1(\mathbb{R}^6)}.
	\end{align*}
	Note $\norm{H_{\alpha,0}}_{L^1}$ is bounded since $K_{\alpha,0}\in C^{\infty}_0(\mathbb{R}^6)$. Then,
	\begin{align*}
	\norm{\mathbf{P}_j R}_{L_x^1}
	&= \norm{\sum_{n_1, n_2}\psi_j(n_1+n_2)m(\frac{|n_1|}{|n_2|})\sum_{|\alpha|=3}C_{\alpha}(n_1,n_2)n_1^{\alpha}\widehat{w}(n_1)\widehat{u}(n_2)e^{ i(n_1+n_2)\cdot x}}_{L^1_x}\\
	&=\norm{\sum_{|\alpha|=3}\sum_{n_1, n_2}K_{\alpha,j}(n_1,n_2)\widehat{S_{j-1}\partial^{\alpha}w}(n_1)\widehat{\widetilde{\mathbf{P}}_ju}(n_2)e^{ i(n_1+n_2)\cdot x}}_{L^1_x}\\
	&=\norm{\sum_{|\alpha|=3}\iint h_{\alpha,j}(y,z)S_{j-1}\partial^{\alpha}w(x-y)\widetilde{\mathbf{P}}_ju(x-z)dydz}_{L^1_x}\\
	&\leq \sum_{|\alpha|=3}\norm{h_{\alpha,j}}_{L^1_{y,z}}\norm{S_{j-1}\partial^{\alpha}w}_{L^1_x}\norm{\widetilde{\mathbf{P}}_ju}_{L^{\infty}_x}\\
	&\lesssim \sum_{|\alpha|=3}2^{j(s-3)}2^{j(3-\rho)}\norm{\partial^{\alpha}w}_{B^{\rho-3}_{1,\infty}}\cdot 2^{-j(s+\beta)}\norm{ u}_{C^{s+\beta}}\\
	&\lesssim 2^{-j(\rho+\beta)}\norm{w}_{B^{\rho}_{1,\infty}}\norm{u}_{C^{s+\beta}}.
	\end{align*}
	Here $\widetilde{\mathbf{P}}_j$ is another Littlewood-Paley projector such that $\widetilde{\mathbf{P}}_j\mathbf{P}_j = \mathbf{P}_j$.
\end{proof}
We can now give the proof of the energy estimate \eqref{ineq:energy_derivative}.
\begin{proof}[Proof of Proposition \ref{theorem:energy_derivative}]
We first write
	\begin{align*}
		&\int_{\T^3}  D^sv(ku^{k-1}D^su-D^su^k) \\
		= &\int_{\T^3}  D^sv \Big[ (ku^{k-1}D^su-kT_{u^{k-1}}D^su)-k(D^s(T_wu)-T_w(D^su))+(kD^sT_{u^{k-1}}u-D^su^k)\Big].
\end{align*}

Using Lemma \ref{low and resonant term lemma}, \ref{lemma:decomposition_lemma}, and \ref{high term lemma}, we estimate the last quantity by
\begin{align*}&\norm{D^sv}_{C^{-\frac{3}{2}-\varepsilon}}\Big( k\norm{u^{k-1}D^su-T_{u^{k-1}}D^su}_{B_{1,1}^{\frac{3}{2}+\varepsilon}} +k\norm{R_1}_{B_{1,1}^{\frac{3}{2}+\varepsilon}}+\norm{kD^sT_{u^{k-1}}u-D^su^k}_{B_{1,1}^{\frac{3}{2}+\varepsilon}}\Big)\\
		&+\left|\int D^sv(F_1+F_2)\right|\\
		\lesssim& \norm{v}_{C^{s-\frac{3}{2}-\varepsilon}}\norm{u}_{C^{s-\frac{1}{2}-\varepsilon}}\norm{u}_{B_{1,\infty}^{2+2\varepsilon}}\norm{u}^{k-2}_{L_x^{\infty}}+\left|\int D^sv(F_1+F_2)\right|\\
		\lesssim& \norm{v}_{C^{s-\frac{3}{2}-\varepsilon}}\norm{u}_{C^{s-\frac{1}{2}-\varepsilon}}\norm{u}_{H^{\sigma}}^{k-1}+\left|\int D^sv(F_1+F_2)\right|.
		\end{align*}
	Here, $R_1, F_1, F_2$ are the terms in Lemma \ref{lemma:decomposition_lemma}. We used (\ref{EQU: Intermediate embeddings}) and (\ref{EQU: Besov embedding}) in the last step.
	
	It remains to deal with $\int D^sv F_1$ and $\int D^sv F_2$. To simplify the notation, in the remaining part of the proof, we use $D$ to represent both $D$ and $\partial$, and fix $\alpha = \frac{3}{2}$. Since $D^{\alpha}$ is self-adjoint, we can write the first integral as
	\begin{align*}
		\int_{\T^3}D^svF_1 = \int_{\T^3}D^{s-\alpha}v(T_{Du^{k-1}}D^{s-1+\alpha}u) + \int_{\T^3}D^{s-\alpha}v[D^{\alpha},T_{Du^{k-1}}]D^{s-1}u =: I_1 + I_2
	\end{align*}
	For $I_1$, since $Du^{k-1}$ has positive regularity(for $s>\frac{5}{2}$), Lemma \ref{low and resonant term lemma} allows us to treat the paraproduct as a real product, which can be bounded by
	\begin{align*}
		I_1 \lesssim \norm{D^{s-\alpha}vD^{s-1+\alpha}u}_{C^{-1-\varepsilon}}\norm{Du^{k-1}}_{B^{1+\varepsilon}_{1,\infty}} \lesssim \norm{D^{s-\alpha}vD^{s-1+\alpha}u}_{C^{-1-\varepsilon}}\norm{u}^{k-1}_{H^{\sigma}}
	\end{align*}
	Moving $D^{\alpha}$ from $v$ to $u$ and using \eqref{ineq: second expansion}, we can decompose $I_2$ into
	\begin{align}\label{ineq: I2}
		I_2 =& \int_{\T^3} D^{s-\alpha}vT_{D^2u^{k-1}}D^{s-2+\alpha}u + \int_{\T^3}D^{s-\alpha}v\widetilde{R},
	\end{align}
	where\footnote{Here we choose $\rho = 1+2\varepsilon$ and $\beta = -1$}
	\begin{align*}
		\norm{\widetilde{R}}_{B_{1,1}^{\frac{3}{2}-\alpha+\varepsilon}}\lesssim\norm{Du^{k-1}}_{B^{1+2\varepsilon}_{1,\infty}}\norm{D^{s-1}u}_{C^{\frac{3}{2}-\varepsilon}} \lesssim \norm{u}_{H^{\sigma}}^{k-1}\norm{u}_{C^{s-\frac{1}{2}-\varepsilon}}.
	\end{align*}
	The first integral in \eqref{ineq: I2}, by Lemma \ref{low and resonant term lemma}, can be treated a real product(for $s>\frac{5}{2}$), then
	\begin{align*}
		\int_{\T^3}D^{s-\alpha}vD^2u^{k-1}D^{s-2+\alpha}u
		&\lesssim\norm{D^{s-\alpha}vD^{s-2+\alpha}u}_{C^{-\varepsilon}}\norm{D^2u^{k-1}}_{B_{1,\infty}^{\varepsilon}}\\
		&\lesssim\norm{D^{s-\alpha}vD^{s-2+\alpha}u}_{C^{-\varepsilon}}\norm{u}_{H^{\sigma}}^{k-1}
	\end{align*}
	For $\int D^svF_2$, note its terms have form
	\begin{align*}
		\int_{\T^3}D^svT_{D^2u^{k-1}}D^{s-2}u = \int_{\T^3} D^{s-\alpha}vT_{D^2u^{k-1}}D^{s-2+\alpha}u + \int_{\T^3}D^{s-\alpha}v\widetilde{R},
	\end{align*}
	where $\tilde{R}$ satisfies the same bound as one in \eqref{ineq: I2}. And it can be treated by exactly the same way as $I_2$.
\end{proof}

\section{Construction of the measure}
In this section, we construct a measure $\vec{\rho}_s$ which is absolutely continuous with respect to $\mu_s$ and corresponds to the formal expression:
\begin{equation}
    d\vec{\rho_{s}}=\mathcal{Z}_{s}^{-1} e^{-\mathscr{E}_{s}(u,v)-E^q(u,v)}dudv.
\end{equation}
Here $\mathscr{E}_s(u,v)$ is the renormalized energy defined in \eqref{eqn: R-energy}, $E(u,v)$ is the Hamiltonian energy \eqref{eqn: H}, and $q=q(s,k)$ is a large integer to be chosen later. 

Define the truncated measures
\begin{align}\label{eqn: truncated-measure}
    \begin{split}
        d\vec{\rho}_{s,N} &= \mathcal{Z}_{s,N,r}^{-1} e^{-\mathscr{E}_{s,N}(u,v)-E^q_N(u,v)}dudv \\
        &= \mathcal{Z}_{s,N}^{-1}e^{-R_{k,s,N}(u)-E^q_N(u,v)}d\nu_s(u,v),
    \end{split}
\end{align}
where the truncated energy $E_N(u,v)$ is defined by
\begin{equation}\label{eqn: truncated-energy}
E_N(u,v):=E(u_N,v_N)=\frac{1}{2}\int_{\mathbb{T}^3}(|\nabla \pi_N u|^2+(\pi_N v)^2)\,\mathrm{d}x+\frac{1}{k+1}\int_{\mathbb{T}^3}(\pi_N u)^{k+1}\,\mathrm{d}x.
\end{equation}

In this section, we prove Proposition \ref{lem: Lp-convergence}, which asserts that the measures $\vec{\rho}_{s,N}$ converge to a limiting measure as $N\rightarrow\infty$.

The general method to establish convergence of the measures is standard (see for example \cite[Remark 3.8]{T2}), and consists of two steps, corresponding Lemma \ref{lem: Lp-convergence} and Proposition \ref{lem: Lp-bound}, respectively.
\begin{enumerate}
\item Convergence of $R_{k,s,N}(u)$ and $E^q_N(u,v)$ in $L^p$. This is a consequence of the regularity properties of the field $\vec{u}$ on the support of $\vec{\mu}_s$, since $R_{k,s,N}(u)$ and $E^q_N(u,v)$ are continuous functions of the Fourier truncated field $\pi_N\vec{u}$.
\item Uniform integrability of $e^{-R_{k,s,N}(u)-E^q_N(u,v)}$ with respect to $\vec{\nu}_s$. This will follow from a uniform bound in $L^p$, $p>1$. It is here that we make use of the variational representation of
\begin{equation}\label{eqn: partition}
\mathcal{Z}_{s,N}:=\mathbb{E}_{\vec{\nu}_s}[e^{-R_{s,k,N}(u)-E^q_N(u,v)}].
\end{equation}
\end{enumerate}
Indeed, the uniform integrability resulting from the second point allows us to take the limit in the expectation
\[\mathbb{E}_{\vec{\nu}_s}[e^{-R_{k,s,N}(u)-E^q_N(u,v)}],\]
which is sufficient to define $\vec{\rho}_s$ as a measure.

Compared to the cubic case, $k=3$, in \cite{GOTW}, the addition of $-E^q(u,v)$ makes the construction of the measure easier as it introduces more decay. Also, as the energy is conserved we have
\[\frac{\mathrm{d}}{\mathrm{d}t}E^q_N(u_N,v_N)=0.\] 
Consequently, no extra terms appear in the energy estimate in section 3.

\begin{definition}
  For $u$ given by \eqref{eqn: useries}, we define
  \[:(D^s u_N)^2: = (D^su_N)^2-\mathbb{E}_{\vec{\nu}_s}[(D^s u_N)^2].\]
This notation is inspired by an analogy with Wick ordering in Gaussian analysis and quantum field theory (see \cite[Chapter 3]{janson}).
\end{definition}

\begin{lemma}\label{lem: Lp-convergence}
Let $s>\frac{3}{2} $. Set 
\begin{equation}\label{eqn: Fskn}
F_{s,k,N}(u,v) := -R_{k,s,N}(u)-E^q_N(u,v).
\end{equation}
Then for $p<\infty$, $F_{s,k,N}$ converges to some $F_{s,k}$ in $L^p(\nu_s)$ as $N\rightarrow \infty$. 
\end{lemma}
\begin{proof}
  Since $s>\frac{3}{2}$ we have by \eqref{eqn: u-C-estimate}, \eqref{eqn: v-C-estimate} that $u\in L^p(\Omega,C^{1+\varepsilon}(\mathbb{T}^3))$ and $v\in L^p(\Omega,C^{\varepsilon}(\mathbb{T}^3))$ for some $\varepsilon>0$. These bounds imply that $E_N(u,v)<\infty$ and moreover is bounded in $L^p$ for any $0<p<\infty$, uniformly in $N$. The same holds for $E_N^q$.

  As in \cite[Proposition 4.3]{GOTW}, we have, for any $p\ge 2$:
  \[\|:(D^s u_N)^2:\|_{L^p(\Omega,C^{-1-\varepsilon}(\mathbb{T}^3))}\le Cp.\]
By duality in $C^\alpha$ spaces \eqref{EQU: Multiplicative estimate}, the term
\[\int Q_{s,N}(u_N) u_N^{k-1}\]
converges in $L^p$.
\end{proof}



\subsection{Variational formulation}
In this section, we apply the Barashkov-Gubinelli variational approach to obtain uniform in $N$ control over the quantity $e^{-R_{k,s,N}(u)-E^q_N(u,v)}$. This is equivalent to showing that the partition function is uniformly bounded, since higher $L^p$ norms of $e^{-R_{k,s,N}(u)-E^q_N(u,v)}$ introduce only constant factors in the representation \eqref{eqn: variational}.

This approach was first applied in \cite{GOTW}. The idea is to write the partition function as an optimization over time-dependent processes, so we begin by representing the measure $\vec{\nu}_s$ as the time 1 distribution of a pair of cylindrical processes. We refer to \cite{BG,GOTW} for more details. 

Let $\Omega=C(\mathbb{R}_+, C^{-\frac{3}{2}-\varepsilon}(\mathbb{T}^3))$. Let $B_n^1(t)$, $B_n^2(t)$, $t\ge 0$ be two sequences of independent standard Brownian motions. We define
\begin{equation}
    \vec{X}(t)=(X_1(t),X_2(t))=\left(\sum\limits_{n\in\Z^3} B^{n}_1(t)e^{i n\cdot x}, \sum\limits_{n\in\Z^3} B^{n}_2(t) e^{i n\cdot x}\right)
\end{equation}
so that $\vec{X}(t)$ is a Brownian motion on $L^2(\T^3)\times L^2(\T^3)$. Set $\vec{Y}(t)=(Y^1(t),Y^2(t))$ where
\begin{equation}
    Y_1(t) = \mathcal{J}^{-s-1}X_1(t) :=  B_1^{0}(t)+\sum\limits_{n\in\Z^3\backslash\{0\}} \frac{B^{n}_1(t)}{(|n|^2+|n|^{2s}+2)^\frac{1}{2}} e^{i n\cdot x}
\end{equation}
and
\begin{equation}
    Y_2(t) = J^{-s}X_2(t) :=  \sum\limits_{n\in\Z^3} \frac{B^{n}_2(t)}{(1+|n|^{2s})^\frac{1}{2}} e^{i n\cdot x}.
\end{equation}
Note that 
\[\textnormal{Law}(\vec{Y}(1)):=\vec{\nu}_s.\]

We let $\mathbb{H}_a$ be the set of progressively measurable processes belonging to 
\[L^2([0,1], L^2(\T^3)\times L^2(\T^3))\] 
almost surely. For $\theta \in \mathbb{H}_a$, the classical Girsanov theorem \cite[Section 5.5]{legall} describes the semimartingale decomposition of $\vec{X}(t)$ and $\vec{Y}(t)$ with respect to the measure $\mathbb{Q}^\theta$ defined by its relative density
\begin{equation}\label{eqn: com}
  \frac{\mathrm{d}\mathbb{Q}^\theta}{\mathrm{d}\mathbb{P}}=e^{\int_0^1 \langle \theta_t,\mathrm{d}X_t\rangle -\frac{1}{2}\int_0^1 \|\theta_s\|^2_{L^2_x}\,\mathrm{d}s}.
  \end{equation}
We have the decompositions
\begin{equation}
    \vec{X}(t)=\vec{X}^\theta(t)+\int_0^t\vec{\theta}(t)\mathrm{d}t
\end{equation}
and 
\begin{equation}
    \vec{Y}(t)=(Y^\theta_1(t), Y^\theta_2(t))+\int_0^t (\mathcal{J}^{-s-1}\theta_1, J^{-s}\theta_2)(t')\,\mathrm{d}t', 
\end{equation}
where $\vec{X}^\theta$ is a $\mathbb{Q}^\theta$ $L^2$-cylindrical Brownian motion and  
\[Y_1^\theta(t):= \mathcal{J}^{-s-1}X_1^\theta(t), \quad Y_2^\theta(t):=J^{-s} X_2^\theta(t).\]
For convenience, we set
\begin{equation}
    \vec{I}^\theta(t):=(I_1^\theta(t), I_2^\theta(t)) = \int_0^t (\mathcal{J}^{-s-1}\theta_1, J^{-s}\theta_2)(t')\,\mathrm{d}t'
\end{equation}
and $\vec{Y}^\theta(t):=(Y^\theta_1(t), Y^\theta_2(t))$. With this notation in place we have the following variational formula for $\mathcal{Z}_{s,N}$.

\begin{lemma}\label{lem: relative-entropy}
Let $\theta\in \mathbb{H}_a$, $N\ge 1$ and let $\mathbb{Q}^\theta$ be the measure defined by \eqref{eqn: com}. Then the \emph{relative entropy} of $\mathbb{Q}^\theta$ with respect to $\mathbb{P}$ is finite:
\[H(\mathbb{Q}^\theta\mid \mathbb{P})=\mathbb{E}\left[\frac{\mathrm{d}\mathbb{Q}^\theta}{\mathrm{d}\mathbb{P}}\log \frac{\mathrm{d}\mathbb{Q}^\theta}{\mathrm{d}\mathbb{P}}\right]<\infty.\]
In particular,
\begin{equation}\label{eqn: quad-finite}
\mathbb{E}_{\mathbb{Q}^\theta}\bigg[\int_0^t \|\theta\|_{L^2_x}\mathrm{d}s\bigg]<\infty.
\end{equation}
\begin{proof}
Once we prove the finiteness of the relative entropy, the bound \eqref{eqn: quad-finite} follows from the inequality \cite[Lemma 2.6]{follmer},
\[\mathbb{E}_{\mathbb{Q}^\theta}\bigg[\int_0^t \|\theta\|_{L^2_x}\mathrm{d}s\bigg]\le 2H(\mathbb{Q}^\theta\mid \mathbb{P}).\]

We turn to the relative entropy. In our case, it takes the following explicit form:
\[H(\mathbb{Q}^\theta\mid \mathbb{P})=\mathbb{E}\left[\frac{e^{-R_{s,k,N}-E^q_N}}{\mathcal{Z}_{s,N}}\log \frac{e^{-R_{s,k,N}-E^q_N}}{\mathcal{Z}_{s,N}}\right].\]
For the partition function $\mathcal{Z}_{s,N}$, we have by Jensen's inequality:
\begin{align}
\mathcal{Z}_{s,N}&=\mathbb{E}[e^{-R_{s,k,N}}e^{-E^q_N}] \nonumber\\
&\ge e^{-\mathbb{E}[R_{s,k,N}]-\mathbb{E}[E^q_N]} \nonumber\\
&\ge c(N). \label{eqn: Z-lower}
\end{align}
In the final step, we have used the integrability of $R_{s,k,N}$ and $E^q_N$, which follows directly from \eqref{eqn: u-C-estimate}, \eqref{eqn: v-C-estimate} since $B^\alpha_{\infty,\infty}\subset L^\infty$ when $\alpha>0$.

Using H\"older's inequality and Young's inequality, it is easy to see that for $q\ge 1$,
\begin{align*}
R_{s,k,N}(Y_1)+E_N^q(\vec{Y})&= \frac{k}{2}\int_{\mathbb{T}} ((D^s  Y_1)^2-\sigma_N^2)(\pi_N Y_1)^{k-1}\,\mathrm{d}x +E_N^q(\vec{Y})\\
&\ge -\frac{k}{2}\sigma_N^2 \int (\pi_N Y_1)^{k-1}\,\mathrm{d}x+ \frac{1}{(k+1)^q}\left(\int (\pi_N Y_1)^{k+1}\,\mathrm{d}x\right)^q\\
&\ge -\frac{k}{2}\sigma_N^2|\mathbb{T}^3|^{\frac{2}{k+1}}\left( \int (\pi_N Y_1)^{k+1}  \right)^\frac{k-1}{k+1} + \frac{1}{(k+1)^q}\left(\int (\pi_N Y_1)^{k+1}\,\mathrm{d}x\right)^q\\
&\ge -\left(\frac{k}{2\varepsilon}\right)^r \frac{ |\mathbb{T}^3|^{\frac{2r}{k+1}}}{r}  \sigma_N^{2r}+(\frac{1}{4}-\frac{\varepsilon^q}{r})\left(\int (\pi_N Y_1)^{k+1}\,\mathrm{d}x\right)^q\\
&> -\infty,
\end{align*}
where $\frac{k-1}{q(k+1)}+\frac{1}{r}=1$. Using \eqref{eqn: Z-lower} and Cauchy-Schwarz, we now have
\[\mathbb{E}\left[\frac{\mathrm{d}\mathbb{Q}^\theta}{\mathrm{d}\mathbb{P}}\log \frac{\mathrm{d}\mathbb{Q}^\theta}{\mathrm{d}\mathbb{P}}\right]\le C(N)\mathbb{E}[e^{-2R_{s,k,N}(Y_1)}+|R_{s,k,N}(Y_1)|^2+1]<\infty.\]
\end{proof}

\end{lemma}

\begin{proposition}
Recall the definition of the partition function $\mathcal{Z}_{s,N}$ 
For $N\in\N$ we have,
\begin{equation}\label{eqn: variational}
    -\log \mathcal{Z}_{s,N}= \inf_{\theta\in \mathbb{H}_a}\E_{\mathbb{Q}^\theta}\left[ R_{k,s,N}(Y_1^{\theta}(1)+I_1^\theta(1))+E^q_N(\vec{Y}^\theta(1)+\vec{I}^\theta(1))  +\frac{1}{2}\int_0^1\norm{\vec{\theta}(t)}_{L^2_x\times L^2_x}^2dt \right].
\end{equation}

\end{proposition}
\begin{proof}
  Given $\theta \in \mathbb{H}_a$, Girsanov's theorem gives:
  \begin{align*}
    -\log \mathcal{Z}_{s,N} &= -\log \mathbb{E}[e^{-R_{k,s,N}-E^q(u_N,v_N)}]\\
    &=-\log \mathbb{E}_{\mathbb{Q}^\theta}[e^{-R_{k,s,N}(Y_1^{\theta}(1)+I_1^\theta(1))-E^q_N(\vec{Y}^\theta(1)+\vec{I}^\theta(1))}e^{-\int_0^1 \langle \theta_t,\mathrm{d}X_t\rangle +\frac{1}{2}\int_0^1 \|\theta_s\|^2_{L^2_x}\,\mathrm{d}s}].
  \end{align*}
 By Jensen's inequality, we have
 \begin{align*}
   -\log \mathcal{Z}_{s,N} \le &~\mathbb{E}_{\mathbb{Q}^\theta}[R_{k,s,N}(Y_1^{\theta}(1)+I_1^\theta(1))-E^q_N(\vec{Y}^\theta(1)+\vec{I}^\theta(1))]\\
   &+\mathbb{E}[\int_0^1 \langle \theta_t,\mathrm{d}X_t\rangle -\frac{1}{2}\int_0^1 \|\theta_s\|^2_{L^2_x}\,\mathrm{d}s]\\
                           = &~\mathbb{E}_{\mathbb{Q}^\theta}[R_{k,s,N}(Y_1^{\theta}(1)+I_1^\theta(1))-E^q_N(\vec{Y}^\theta(1)+\vec{I}^\theta(1))]\\
   &+\mathbb{E}[\int_0^1 \langle \theta_t,\mathrm{d}X_t^\theta \rangle +\frac{1}{2}\int_0^1 \|\theta_s\|^2_{L^2_x}\,\mathrm{d}s].
 \end{align*}
 
 If
 \[\mathbb{E}_{\mathbb{Q}^\theta}\bigg[\int_0^1 \|\theta_s\|^2_{L^2_x}\,\mathrm{d}s\bigg]<\infty,\]
 the stochastic integral term is a martingale, so its expectation vanishes and we find
 \begin{equation}\label{eqn: var-bound-1}
   -\log \mathcal{Z}_{s,N} \le \mathbb{E}_{\mathbb{Q}^\theta}\bigg[R_{k,s,N}(Y_1^{\theta}(1)+I_1^\theta(1))-E^q_N(\vec{Y}^\theta(1)+\vec{I}^\theta(1))+\frac{1}{2}\int_0^1 \|\theta_s\|^2_{L^2_x}\,\mathrm{d}s\bigg].
 \end{equation}
 If instead
\[\mathbb{E}_{\mathbb{Q}^\theta}\bigg[\int_0^1 \|\theta_s\|^2_{L^2_x}\,\mathrm{d}s\bigg]=\infty,\]
the inequality \eqref{eqn: var-bound-1} holds trivially provided we verify that
\[R_{k,s,N}(Y_1^{\theta}(1)+I_1^\theta(1))-E^q_N(\vec{Y}^\theta(1)+\vec{I}^\theta(1))\]
is $\mathbb{Q}^\theta$-integrable, which we do below.

 Conversely, the measure
 \[\frac{\mathrm{d}\mathbb{Q}^N}{\mathrm{d}\mathbb{P}}=\frac{e^{-R_{k,s,N}(Y_1)-E^q_N(\vec{Y}(1))}}{\mathcal{Z}_{s,N}}\]
 is absolutely continuous with respect to $\mathbb{P}$, so there is a $\tilde{\theta}\in \mathbb{H}_a$, such that
  \[\frac{\mathrm{d}\mathbb{Q}^N}{\mathrm{d}\mathbb{P}}=e^{\int_0^1 \langle \tilde{\theta}^N_t,\mathrm{d}X_t\rangle-\frac{1}{2}\int_0^1\|\tilde{\theta}^N_t\|^2_{L^2_x}\,\mathrm{d}t}.\]
  Combining the last two expressions gives
  \begin{equation}\label{eqn: take-exp}
    -\log\mathcal{Z}_{s,N}:=R_{k,s,N}(Y_1)+E^q_N(\vec{Y}(1))+\int_0^1 \langle \tilde{\theta}^N_t,\mathrm{d}X_t\rangle-\frac{1}{2}\int_0^1\|\tilde{\theta}^N_t\|^2_{L^2_x}\,\mathrm{d}t.
    \end{equation}
By Lemma \ref{lem: relative-entropy} we have
   \[\mathbb{E}_{\mathbb{Q}^{\tilde{\theta}}}\bigg[\int_0^1 \|\tilde{\theta}_s\|^2_{L^2_x}\,\mathrm{d}s\bigg]<\infty,\]
   we can take expectations in \eqref{eqn: take-exp} and the martingale term vanishes, so
   \[-\log \mathcal{Z}_{s,N}=\mathbb{E}\left[R_{k,s,N}(Y_1)+E^q_N(\vec{Y}(1))+\frac{1}{2}\int_0^1\|\tilde{\theta}^N_t\|^2_{L^2_x}\,\mathrm{d}t\right].\]
\end{proof}

\subsection{Exponential integrability}
We now prove Proposition \ref{lem: Lp-bound} by estimating the quantity on the right side of \eqref{eqn: variational}. Since the time $t=1$ is fixed, for simplicity we set 
\[\vec{Y}^\theta:=(Y_1^\theta,Y_2^\theta)=(Y_1^\theta(1), Y_2^\theta(1)).\] 
A simple application of Young's inequality gives
\begin{equation}
    \frac{1}{2}E^q_N(\vec{I}^\theta)\leq C E^q_N(\vec{Y}^\theta)+E^q_N(\vec{Y}^\theta+\vec{I}^\theta)\nonumber.
\end{equation}
for some large constant $C$. Hence it suffices to bound
\begin{equation}\label{EQU: Need to bound below 2}
        \E_{\mathbb{Q}^\theta}\left[ R_{k,s,N}(Y_1^{\theta}+I_1^\theta)-CE^q_N(\vec{Y}^\theta)+\frac{1}{2}E_N^q(\vec{I}^\theta)  +\frac{1}{2}\int_0^1\norm{\vec{\theta}(t)}_{L^2_x\times L^2_x}^2dt \right].
\end{equation}

The following lemma gives the regularity of  $D^sY_1^\theta$, $:(D^sY_1^\theta)^2:$ and $Y^\theta_1$. 
\begin{lemma}\label{Lemma: prob control on Y}
Let $2<p<\infty$. Then for $\eps>0$,
\begin{equation}
    \sup\limits_{\theta\in \mathbb{H}_a}\E_{\mathbb{Q}^\theta}\left[\norm{D^sY_1^\theta}_{C^{-\frac{1}{2}-\eps}}^p + \norm{:(D^sY_1^\theta)^2:}_{C^{-1-\eps}}^p + \norm{Y^\theta_1}_{C^{s-\frac{1}{2}-\eps}}^p\right]<\infty.\nonumber 
\end{equation}
\begin{proof}
This follows directly from Proposition 4.3 in \cite{GOTW}.
\end{proof}
\end{lemma}

A direct application of Cauchy-Schwarz (see \cite[Lemma 4.7]{GOTW}) gives
\begin{equation}
    \norm{I_1^\theta}_{H^{s+1}}^2\leq \int_0^1\norm{\theta_t}_{L^2}^2dt.
\end{equation}

\begin{lemma}
For $s>\frac{3}{2}$,
\begin{equation}
   \sup\limits_{\theta\in \mathbb{H}_a}\E_{\mathbb{Q}^\theta}\left[E_N^q(\vec{Y}^\theta)  \right] <\infty\nonumber
\end{equation}
independently of $N$.
\end{lemma}
\begin{proof}
Under $\mathbb{Q}^\theta$, $\vec{Y}^\theta(1)=(Y_1^\theta(1),Y_2^\theta(1))$ has the same distribution as the pair $(u,v)$ in under $\vec{\nu}_s$. The result then follows from \eqref{eqn: u-C-estimate}, \eqref{eqn: v-C-estimate}.
\end{proof}

We introduce some abbreviated notations for the most common terms appearing in the estimates below. We set:
\begin{align*}
Y&:=Y_1^\theta,\\
\Theta&:=I_1^\theta,\\
E&:= E_N(\vec{I}^\theta).
\end{align*} 
From the definition of $R_{k,s,N}$ we have
\begin{align}
    R_{k,s,N}(Y+\Theta)=&\frac{k(k-1)}{2}\int_{\T^3}:(D^sY)^2: \sum\limits_{m=0}^{k-1}\binom{k-1}{m} Y^{k-1-m}\Theta^m\nonumber \\ 
    &+k(k-1)\int_{\T^3}D^sYD^s\Theta \sum\limits_{m=0}^{k-1}\binom{k-1}{m} Y^{k-1-m}\Theta^m\nonumber \\ 
    &+\frac{k(k-1)}{2}\int_{\T^3}(D^s\Theta)^2\sum\limits_{m=0}^{k-1}\binom{k-1}{m} Y^{k-1-m}\Theta^m\nonumber \\
    &+\frac{1}{k+1}\int_{\T^3}(Y+\Theta)^{k+1}.
\end{align}
We aim to bound \eqref{EQU: Need to bound below 2} by using Young's inequality and the positive terms
\begin{equation}
    \int_{\T^3}  (D^s\Theta)^2\Theta^{k-1},\quad \int_{\T^3}  \Theta^{k+1},\quad E^q,\quad \norm{\Theta}_{H^{s+1}}^2\nonumber
\end{equation}
in \eqref{EQU: Need to bound below 2}.

\begin{lemma}
(Terms quadratic in $D^sY$). Let $s>\frac{3}{2}$ and $0\leq m\leq k-1$. Then for sufficiently small $\delta>0$ there exists small $\eps>0$ and large $p$ and $c(\delta)$ such that
\begin{align*}
    \left|\int_{\T^3}  :(D^sY)^2: Y^{k-1-m}\Theta^m \right| \leq &~c(\delta)\left(\norm{:(D^sY)^2:}_{C^{-1-\eps}}^p+ \norm{D^sY}_{C^{-\frac{1}{2}-\eps}}^p +\norm{Y}_{C^{s-\frac{1}{2}-\eps}}^p \right)\\
    &+\delta\left(\norm{\Theta}_{H^{s+1}}^2 +E^q \right).
\end{align*}
\end{lemma}
\begin{proof}
For $m=0$, using \eqref{EQU: Duality} and \eqref{EQU: Intermediate embeddings} and \eqref{EQU: Algebra property} we have
\begin{align*}
    \left|\int_{\T^3}  :(D^sY)^2: Y^{k-1}\right|&\lesssim \norm{:(D^sY)^2:}_{C^{-1-\eps}}\norm{Y^{k-1}}_{B^{1+2\eps}_{1,\infty}}\\
    &\lesssim \norm{:(D^sY)^2:}_{C^{-1-\eps}}\norm{Y}_{C^{1+2\eps}}^{k-1}\\
    &\lesssim \norm{:(D^sY)^2:}_{C^{-1-\eps}}\norm{Y}_{C^{s-\frac{1}{2}-\eps}}^{k-1}
\end{align*}
if $s>\frac{3}{2}$ and $\eps>0$ is small. The estimate then  follows from Young's inequality. If $m=k-1$, using \eqref{EQU: Duality}, \eqref{EQU: Intermediate embeddings}, \eqref{EQU: Fracational Leibniz rule cor} and \eqref{EQU: Besov embedding} we have
\begin{align*}
      \left|\int_{\T^3}  :(D^sY)^2: \Theta^{k-1}\right|&\lesssim \norm{:(D^sY)^2:}_{C^{-1-\eps}}\norm{\Theta^{k-1}}_{B^{1+2\eps}_{1,\infty}}\\
      &\lesssim \norm{:(D^sY)^2:}_{C^{-1-\eps}}\norm{\Theta^{k-2}}_{L^1}\norm{\Theta}_{B^{1+2\eps}_{\infty,\infty}}\\
      &\lesssim \norm{:(D^sY)^2:}_{C^{-1-\eps}}\norm{\Theta^{k+1}}_{L^1}^{\frac{k-2}{k+1}}\norm{\Theta}_{B^{\frac{5}{2}+3\eps}_{2,2}}\\ 
      &\lesssim \norm{:(D^sY)^2:}_{C^{-1-\eps}}E^{\frac{k-2}{k+1}}\norm{\Theta}_{H^{s+1}}
\end{align*}
if $s>\frac{3}{2}$ and $\eps>0$ is small. If we choose $q$ large enough so that
\begin{equation}
    \frac{k-2}{q(k+1)}+\frac{1}{2}<1
\end{equation}
the stated inequality then follows from Young's inequality.

If $0<m<k-1$ then similar to the above,
\begin{align*}
    \left|\int_{\T^3}  :(D^sY)^2: Y^{k-1-m}\Theta^m\right| & \lesssim \norm{:(D^sY)^2:}_{C^{-1-\eps}}\norm{Y^{k-1-m}\Theta^m}_{B_{1,\infty}^{1+2\eps}}\\
    &\lesssim  \norm{:(D^sY)^2:}_{C^{-1-\eps}}\norm{Y^{k-1-m}}_{C^{1+2\eps}}\norm{\Theta^m}_{B^{1+2\eps}_{1,\infty}}\\
    &\lesssim \norm{:(D^sY)^2:}_{C^{-1-\eps}}\norm{Y}_{C^{s-\frac{1}{2}-\eps}}^{k-1-m}E^{\frac{m-1}{k+1}}\norm{\Theta}_{H^{s+1}}.
\end{align*}
Hence if we choose $q$ large enough so that 
\begin{equation}
    \frac{m-1}{q(k+1)}+\frac{1}{2}<1\nonumber.
\end{equation}
Young's inequality the gives the desired result. 
\end{proof}

\begin{lemma}
(Terms linear in $D^sY$). Let $s>\frac{5}{2}$ and $0\leq m\leq k-1$. Then for sufficiently small $\delta>0$ there exists small $\eps>0$ and large $p$ and $c(\delta)$ such that
\begin{align*}
    \left|\int_{\T^3}  D^sYD^s\Theta Y^{k-1-m}\Theta^m \right|\leq &~c(\delta) \left( \norm{D^sY}_{C^{-\frac{1}{2}-\eps}}^p+\norm{Y}_{C^{s-\frac{1}{2}-\eps}}^p  \right)\\
    &+\delta\left(\norm{\Theta}_{H^{s+1}}^2 +E^q\right).
\end{align*}
\end{lemma}
\begin{proof}

First we estimate the term corresponding to $m=k-1$. Using \eqref{EQU: Duality}, \eqref{EQU: Intermediate embeddings} followed by \eqref{EQU: Fractional Leibniz rule},
\begin{align}
    \left|\int_{\T^3}  D^sY D^s\Theta \Theta^{k-1}\right| \lesssim \norm{D^sY}_{C^{-\frac{1}{2}-\eps}}\norm{D^s\Theta \Theta^{k-1}}_{B^{\frac{1}{2}+2\eps}_{1,\infty}}.\nonumber
\end{align}
Hence it remains to estimate $\norm{D^s\Theta \Theta^{k-1}}_{B^{\frac{1}{2}+2\eps}_{1,\infty}}$. Using \eqref{EQU: Fractional Leibniz rule}, \eqref{EQU: Intermediate embeddings}, \eqref{EQU: Besov embedding}, \eqref{EQU: Fracational Leibniz rule cor}, \eqref{EQU: Besov embedding} again, Jensen's inequality and \eqref{EQU: Interpolation}, we have
\begin{align}\label{EQU: Ds theta theta k-1}
\norm{D^s\Theta \Theta^{k-1}}_{B^{\frac{1}{2}+2\eps}_{1,\infty}} &\lesssim \norm{D^s\Theta}_{B^{\frac{1}{2}+2\eps}_{2,\infty}}\norm{\Theta^{k-1}}_{L^2}+\norm{D^s\Theta}_{L^2}\norm{\Theta^{k-1}}_{B^{\frac{1}{2}+2\eps}_{2,\infty}}\nonumber \\
&\lesssim \norm{D^s\Theta}_{B^{\frac{1}{2}+2\eps}_{2,\infty}}\norm{\Theta^{k-1}}_{B^{\frac{1}{2}+2\eps}_{2,\infty}}\nonumber \\
&\lesssim \norm{\Theta}_{H^{s+\frac{1}{2}+2\eps}}\norm{\Theta^{k-1}}_{B^{2+2\eps}_{1,\infty}}\nonumber \\
&\lesssim \norm{\Theta}_{H^{s+\frac{1}{2}+2\eps}}\norm{\Theta^{k-2}}_{L^1}\norm{\Theta}_{B^{2+2\eps}_{\infty,\infty}}\nonumber \\
&\lesssim \norm{\Theta}_{H^{s+\frac{1}{2}+2\eps}}\norm{\Theta^{k+1}}_{L^1}^\frac{k-2}{k+1}\norm{\Theta}_{B^{\frac{7}{2}+2\eps}_{2,\infty}}\nonumber \\
&\lesssim \norm{\Theta}_{H^{s+1}}^\gamma \norm{\Theta}_{L^2}^{1-\gamma}\norm{\Theta^{k+1}}_{L^1}^\frac{k-2}{k+1}\norm{\Theta}_{H^{s+1}}\nonumber \\
&\lesssim \norm{\Theta}_{H^{s+1}}^{1+\gamma} E^{1-\gamma+\frac{k-2}{k+1}}
\end{align}
for $s>\frac{5}{2}$ and $\eps>0$ small where $\gamma=\gamma(s,\eps)=\frac{s+\frac{1}{2}+2\eps}{s+1}<1$. If we choose $\eps=\eps(s,k)$ small enough and $q=q(s,k)$ large enough so that 
\begin{equation}
    \frac{1+\gamma}{2}+\frac{1}{q}\left( 1-\gamma+\frac{k-2}{k+1} \right)<1 \nonumber
\end{equation}
the desired inequality follows from Young's inequality. 

Now for the case $0<m<k-1$, using \eqref{EQU: Duality}, \eqref{EQU: Intermediate embeddings},\eqref{EQU: Fractional Leibniz rule} we have,
\begin{align*}
   \left| \int_{\T^3}  D^sY D^s\Theta Y^{k-1-m}\Theta^{m}\right| &\lesssim \norm{D^sY}_{C^{-\frac{1}{2}-\eps}}\norm{Y^{k-1-m}D^s\Theta \Theta^{m}}_{B^{\frac{1}{2}+2\eps}_{1,\infty}}\\
    &\lesssim \norm{D^sY}_{C^{-\frac{1}{2}-\eps}}\norm{Y^{k-1-m}}_{C^{\frac{1}{2}+2\eps}}\norm{D^s\Theta \Theta^{m}}_{B^{\frac{1}{2}+2\eps}_{1,\infty}}\\
    &\lesssim \norm{D^sY}_{C^{-\frac{1}{2}-\eps}}\norm{Y}_{C^{s-\frac{1}{2}-\eps}}^{k-1-m}\norm{D^s\Theta \Theta^{m}}_{B^{\frac{1}{2}+2\eps}_{1,\infty}}
\end{align*}
for $s>1$ and $\eps>0$ small.
It remains to estimate the term $\norm{D^s\Theta \Theta^{m}}_{B^{\frac{1}{2}+2\eps}_{1,\infty}}$. If $s>\frac{5}{2}$ and $\eps>0$ is small enough, this term can be estimated in a manner similar to \eqref{EQU: Ds theta theta k-1}. 

Finally we estimate the term corresponding to $m=0$. We have,
\begin{align*}
    \left|\int_{\T^3}  D^sY D^s\Theta Y^{k-1-m}\right| &\lesssim \norm{D^sY}_{C^{-\frac{1}{2}-\eps}}\norm{D^s\Theta Y^{k-1}}_{B^{\frac{1}{2}+2\eps}_{1,\infty}}\\
    &\lesssim \norm{D^sY}_{C^{-\frac{1}{2}-\eps}}\norm{Y}_{C^{\frac{1}{2}+2\eps}}^{k-1}\norm{D^s\Theta}_{B^{\frac{1}{2}+\eps}_{1,\infty}}\\
    &\lesssim \norm{D^sY}_{C^{-\frac{1}{2}-\eps}}\norm{Y}_{C^{s-\frac{1}{2}+2\eps}}^{k-1}\norm{D^s\Theta}_{H^{s+1}}
\end{align*}
for $s>1$ and $\eps>0$ small. The desired inequality then follows from Young's inequality.

\end{proof}

\begin{lemma}
(Terms quadratic in $D^s\Theta$).  Let $s>\frac{1}{2}$ and $0< m< k-1$. Then for sufficiently small $\delta>0$ there exists small $\eps>0$ and large $p$ and $c(\delta)$ such that
\begin{align}
    \left|\int_{\T^3}  (D^s\Theta)^2 Y^{k-1-m}\Theta^m\right| \leq &c(\delta) \left( \norm{D^sY}_{C^{-\frac{1}{2}-\eps}}^p +\norm{Y}_{C^{s-\frac{1}{2}+2\eps}}^p \right)\nonumber\\
    &+\delta\left(\norm{\Theta}_{H^{s+1}}^2 +\norm{D^s\Theta \Theta^{\frac{k-1}{2}}}_{L^2}^2 +E^q\right).\nonumber
\end{align}
\end{lemma}
\begin{proof}
Using Young's inequality,
\begin{align*}
    \left|\int_{\T^3}  (D^s\Theta)^2 Y^{k-1-m}\Theta^m \right| \leq C(\delta)\int_{\T^3}  (D^s\Theta)^2 Y^{k-1}+ \delta\int_{\T^3}  (D^s\Theta)^2 \Theta^{k-1}.
\end{align*}
It remains to estimate the first term on the right hand side of the above inequality. Using H\"older's inequality, \eqref{EQU: Interpolation} and the fact that $s>\frac{1}{2}$ we have,
\begin{align*}
    \left|\int_{\T^3}  (D^s\Theta)^2 Y^{k-1}\right| & \lesssim \norm{\Theta}_{H^s}^2\norm{Y^{k-1}}_{L^\infty}\\
    & \lesssim \norm{\Theta}_{H^{s+1}}^{2\frac{s}{s+1}}\norm{\Theta}_{L^2}^{2\frac{1}{s+1}}\norm{Y}^{k-1}_{L^\infty}\\
    &\lesssim \norm{\Theta}_{H^{s+1}}^{\frac{2s}{s+1}}\norm{\Theta}_{L^{k+1}}^{\frac{2}{s+1}}\norm{Y}^{k-1}_{C^{s-\frac{1}{2}-\eps}}\\
    &\lesssim \norm{\Theta}_{H^{s+1}}^{\frac{2s}{s+1}}E^{\frac{2}{(s+1)(k+1)}}\norm{Y}^{k-1}_{C^{s-\frac{1}{2}-\eps}}.
\end{align*}
For $q$ large enough,
\begin{equation}
    \frac{s}{s+1}+\frac{2}{q(s+1)(k+1)}<1 \nonumber
\end{equation}
and so the desired inequality follows from Young's inequality.
\end{proof}

\begin{lemma}
(Remaining Terms) Let $s>\frac{1}{2}$. Then for sufficiently small $\delta>0$ there exists small $\eps>0$ and large $c(\delta)$ such that
\begin{align}
\int_{\mathbb{T}^3} (Y+\Theta)^{k+1}\leq C(\delta)\norm{Y}_{C^{s-\frac{1}{2}-\eps}}+\delta\int_{\T^3}  \Theta^{k+1}.\nonumber
\end{align}
\end{lemma}
\begin{proof}
Using Young's inequality and \eqref{EQU: Intermediate embeddings} we have,
\begin{align}
\int_{\mathbb{T}^3} (Y+\Theta)^{k+1}&\leq C(\delta)\int_{\mathbb{T}^3} Y^{k+1}+\delta\int_{\mathbb{T}^3} \Theta^{k+1}\nonumber\\
&\leq C(\delta)\norm{Y}_{C^{s-\frac{1}{2}-\eps}}^{k+1}+\delta\int_{\mathbb{T}^3} \Theta^{k+1},
\end{align}
which completes the proof.
\end{proof}

\section{Dispersionless case}\label{sec: dispersionless}
In this section, we show that the dispersion is essential to the quasi-invariance result, by showing that it fails if the Laplacian term is absent from the system \eqref{NLW}. More precisely, we show that there exists a dense sequence of times $t$ such that the distribution of the flow of the dispersionless model \eqref{eqn: ODE} at time $t$ is not absolutely continuous with respect to the distribution of the initial data. This answers a question posed in the introductions to \cite{OT}, \cite{GOTW}. 

The proof uses an idea developed by Oh, Tzvetkov and the third author in \cite{OST} to prove the same result for a Schr\"odinger-type equation, using almost-sure properties of the series \eqref{series}. Unlike in \cite{OST} situation, no explicit solution of the ODE \eqref{eqn: ODE} is available, so we instead use the invariance of the Hamiltonian to derive a contradiction.

The dispersionless system is 
\begin{equation}\label{eqn: ODE}
\begin{cases}
\partial_t u=v,\\
\partial_t v= -u^k
\end{cases}.
\end{equation}
We take the initial data
\[(u(0,x),v(0,x)):=(u^\omega,v^\omega),\]
where $u^\omega$, $v^\omega$ are the random series already given in \eqref{series}
\begin{equation}\label{eqn: uv-rep}
\begin{split}
u^\omega (x)&=\sum_{n\in \mathbb{Z}^3}\frac{g_n}{\langle n\rangle^{s+1}}e^{in\cdot x},\\
v^\omega (x)&=\sum_{n\in \mathbb{T}^3}\frac{h_n}{\langle n\rangle^s} e^{in\cdot x}.
\end{split}
\end{equation}

Our main tool to derive Theorem \ref{thm: dispersionless}, the law of the iterated logarithm, gives a fine description of the regularity of the process at a \emph{fixed point} $x\in \mathbb{T}^d$. The key point is that the result holds almost surely, so it must also hold on the support of any measure that is mutually absolutely continuous with respect to $\mu_s$.

The analog for Gaussian fields indexed by $\mathbb{R}^3$ the following result was proved in \cite[Theorem 1.3]{BJR}. Their result is more general and covers non-stationary Gaussian fields whose covariance is defined by a pseudodifferential operator. The proof uses a wavelet decomposition of the process. Given the explicit representations \eqref{series}, they can also be derived more directly using more classical tools. This was done in the case of one dimensional Gaussian fields in \cite{OST}. We do not reproduce the details of the proof here.
\begin{proposition}
For $x\in \mathbb{T}^3$, let $v(x)$ be given by the random series defined in \eqref{eqn: uv-rep}. We have the following:
\begin{enumerate}
\item For $s\in (\frac{1}{2},\frac{3}{2})$, 
\[\limsup_{|h|\rightarrow 0} \frac{v(x+h)-v(x)}{\sqrt{c_s |h|^{2s-1}\log\log \frac{1}{|h|}}}=1\]
almost surely.
\item For $s=\frac{3}{2}$,
\[\limsup_{|h|\rightarrow 0} \frac{v(x+h)-v(x)}{\sqrt{c_s |h|^2\log \frac{1}{|h|}\log\log\log  \frac{1}{|h|}}}=1.\]
\item When $s>\frac{3}{2}$, let $r=\lfloor s-\frac{1}{2}\rfloor$, then $\partial_x^r v$ exists and is continuous, and satisfies the above LILs with $v$ replaced by $\partial_x^r v$ and $s$ replaced by $s-r$.
\end{enumerate}
\end{proposition}

The Hamiltonian associated to the ODE is
\begin{equation}
H(u,v)=\frac{1}{2}v^2(x)+\frac{1}{k+1}u^{k+1}(x).
\end{equation}
This quantity is conserved along the flow of the ODE \eqref{eqn: ODE}:
\begin{equation}\label{eqn: conserved}
H(u(0),v(0))=H(u(t),v(t))
\end{equation}
for each $t\ge 0$.

\begin{proof}[Proof of Theorem \ref{thm: dispersionless}]
Let $h_n$, $|h_n|\downarrow 0$ be a sequence along which the $\limsup$ in the LIL for $v(t,x)$ is attained.  Differentiating both sides of this equation $r$ times, we find:
\begin{equation}\label{eqn: equality}
v(0,x)\partial_x^rv(0,x)+u(0,x)^k\partial_x^r u(0,x)+\text{l.o.t.}=v(t,x)\partial_x^rv(t,x)+u(t,x)^k\partial_x^r u(t,x)+\text{l.o.t.},
\end{equation}
where ``l.o.t.'' denotes lower order terms which are more regular that $\partial_x^r v(0,x)$, $\partial_x^r v(t,x)$. The equality \eqref{eqn: equality} also holds with $x$ shifted by $h_n$ on both sides, so that after subtraction and division by 
\[\sqrt{c_s |h_n|^{2s-1}\log\log \frac{1}{|h_n|}}\]
in case $s-r\neq 3/2$ and
\[\sqrt{c_s |h_n|^2\log \frac{1}{|h_n|}\log\log \frac{1}{|h_n|}}\]
if $s-r=\frac{3}{2}$, we find, for each $x\in \mathbb{T}^3$ and $t\ge 0$:
\begin{equation}
v(0,x)\ge v(t,x)
\end{equation}
almost surely.
Exchanging the roles of $v(0,x)$ and $v(t,x)$, we obtain
\begin{equation}\label{eqn: static}
v^\omega(x)=v(0,x)=v(t,x)
\end{equation}
with probability 1.
Using \eqref{eqn: conserved} again, we now find
\[\frac{1}{k+1}u^{k+1}(t,x)=H(u(0),v(0))-\frac{1}{2}v(t,x)=\frac{1}{k+1}u^{k+1}(0,x),\]
so
\[u(t,x)=\pm u(0,x),\]
almost surely. Let $H_0=H(u(0),v(0))$. The equality
\[(u(t,x),v(t,x))=(\pm u(0,x),v(0,x))\]
implies that $t$ is equal to the period $T$ of the system \eqref{eqn: ODE} on the energy surface $H(u,v)=H_0$:
\begin{equation}\label{eqn: T}
T:=\frac{4\cdot (2H_0)^{\frac{1}{2}-\frac{1}{k+1}} }{(k+1)^{1/(k+1)}} \int_0^1\frac{1}{(1-y^2)^{1/(k+1)}}\mathrm{d}y,
\end{equation}
or equal to an integer multiple of $T$ plus or minus
\begin{equation}\label{eqn: T-diff}
\begin{split}
\Delta&:=\frac{2}{(k+1)^{1/(k+1)}}\int_{|v(0,x)|}^{\sqrt{2H_0}}\frac{1}{(H_0-\frac{y^2}{2})^{\frac{1}{k+1}}}\,\mathrm{d}y\\
&=\frac{2\cdot (2H_0)^{\frac{1}{2}-\frac{1}{k+1}}}{(k+1)^{1/(k+1)}}\int_{\frac{|v(0,x)|}{\sqrt{2H_0}}}^1 \frac{1}{(1-y^2)^{\frac{1}{k+1}}}\,\mathrm{d}y.
\end{split}
\end{equation}
Both quantities \eqref{eqn: T} and \eqref{eqn: T-diff} have continuous distributions if $k\neq 1$, so 
\[\mathbb{P}(t= n\cdot T + m\cdot \Delta \text{ for some } n\in \mathbb{Z}_+, m\in \{-1,0,1\})=0,\]
so the assumption of absolute continuity between the distributions is untenable.
\end{proof}

\textbf{Acknowledgements} The authors would like to thank Tadahiro Oh for facilitating our collaboration, and explaining the argument in Appendix \ref{sec: critical}. P.S. is supported by NSF Grant DMS-1811093. W. J. T. was supported by The Maxwell Institute Graduate School in Analysis and its Applications, a Centre for Doctoral Training funded by the UK Engineering and Physical Sciences Research Council (grant EP/L016508/01), the Scottish Funding Council, Heriot-Watt University and the University of Edinburgh.

\appendix
\section{Bound for the $H^{\sigma}$ in the critical case}\label{sec: critical}
In this section, we derive the growth bound \eqref{eqn: bound-2} in the case $d=3$, $k=5$. This is the energy-critical nonlinearity. In this case, it is known solutions exist globally (see \cite[Chapter 5]{tao}, \cite[Chapter V]{sogge}) and scatter. \marginpar{references}

We begin by establishing a bound for the solutions of the equation on $\mathbb{R}^3$ instead of $\mathbb{T}^3$. We have the global space time bound
\begin{equation}\label{eqn: L5}
\|u\|_{L^5_t L^{10}_x}\le C(H_0),
\end{equation}
where $H_0$ is the initial energy. See \cite[Chapter V, Proposition 3.1]{sogge} for a global bound in $L^4_tL^{12}_x$. See also \cite[Theorem 1.1]{tao-st} for a quantitive estimate in terms of $H_0$. The estimate \eqref{eqn: L5} follows from this by the standard Strichartz estimates for the wave equation (\cite[Chapter IV, Corollary 1.2]{sogge} with $q=5$, $r=2$, $\tilde{q}'=1$, $\tilde{r}'=2$).

Let $\eta>0$, and divide $\mathbb{R}_+$ into a finite number of interval $I_j$, $j=1,\ldots, J$ such that
\[\|u\|_{L^5_t(I_j;L^{10}_x)}\le \eta.\]
Denote by $t_j^-$ the left endpoint of the interval $I_j$, so that
\begin{align*}
  I_j&=[t_j^-,t_j^+),\\
  t_{j+1}^-&=t_j^+.
\end{align*}
We now recall the Strichartz estimate for the linear wave equation. A pair of exponents $(q,r)$, $q\ge 2$, $2\le r<\infty$, is called $s$-admissible (in dimension 3) if
\begin{align*}
\frac{1}{q}+\frac{1}{r}&\le \frac{1}{2},\\  
\frac{1}{q}+\frac{3}{r}&=\frac{3}{2}-s.
\end{align*}
Let $s>0$. If $(q,r)$ and $(\tilde{q},\tilde{r})$ are $s$--admissible (respectively, $1-s$--admissible pairs, then if $(\tilde{u},\tilde{v})$ solves the linear wave equation with right hand side given by $F$, we have \cite[Theorem 2.6]{tao}
\[\|(\tilde{u},\tilde{v})\|_{L_t^\infty(I;\vec{H}^s(\mathbb{R}^3))}+\|\tilde{u}\|_{L_t^q(I;L^r_x)}\lesssim\|(\tilde{u}_0,\tilde{v}_0)\|_{\vec{H}^s(\mathbb{R}^3)}+\|F\|_{L_t^{\tilde{q}'}(I;L^{\tilde{r}'}_x)}.\]
We choose $p=5$, $r=10$, $\tilde{q}'=1$, $\tilde{r}'=2$. We let choose $\tilde{u}=\nabla^{\sigma-1} u$, we obtain
\[\|(u,v)\|_{C_t(I;\mathcal{H}^{\sigma})}+\|u\|_{L^5_t(I; W^{\sigma-1,10}_x)}\le C\|(u_0,v_0)\|_{\mathcal{H}^{\sigma}}+C\|\nabla^{\sigma-1}(|u|^4u)\|_{L^1_t(I;L^2_x)}.\]
Applying the Leibniz rule \eqref{EQU: Fracational Leibniz rule cor} with $s=\sigma-1$, $p_0=10$, $p_1=5/4$ and $p=q=2$ to the final term, we obtain, on each interval $I_j$, the estimate
\begin{align*}
\|u\|_{C_t(I_j;H^{\sigma})}+\|u\|_{L^5_t(I_j; W^{\sigma-1,10}_x)} &\le C \big(\| (u, v)(t_j^-)\|_{\mathcal{H}^\sigma} + \| u \|_{L^5_t(I_j; W^{\sigma-1,
10}_x)} \| u \|_{L^5_t(I_j;L^{10}_x)}^4\big)\\
&\le C \big(\| (u, v)(t_j^-)\|_{\mathcal{H}^\sigma} + \eta^4 \| u \|_{L^5_t(I_j; W^{\sigma-1,
10}_x)} \big).
\end{align*}
For $\eta$ small enough, we obtain
\[\|(u,v)\|_{C_t(I_j;\mathcal{H}^{\sigma})}+\|u\|_{L^5(I_j;W^{\sigma-1,10}_x)}\le C'\|(u,v)(t_j^-)\|_{\mathcal{H}^{\vec{\sigma}}(\mathbb{R}^3)}.\]
Repeating this over each of the $J$ intervals, we obtain the bound
\[\|(u,v)\|_{C_t(\mathbb{R}_+;\mathcal{H}^{\sigma})}+\|u\|_{L^5(\mathbb{R}_+;W^{\sigma-1,10}_x(\mathbb{R}^3))}\le (C')^J.\]
The same argument applied to the negative time direction finally gives
\begin{equation}\label{eqn: L5}
\|(u,v)\|_{C_t(\mathbb{R};\mathcal{H}^{\sigma})}+\|u\|_{L^5(\mathbb{R};W^{\sigma-1,10}_x(\mathbb{R}^3))}\le (C')^J.
\end{equation}

To transfer the estimate \eqref{eqn: L5} to $\mathbb{T}^3\approx [-\frac{1}{2},\frac{1}{2})^3$, we consider initial data $(u_0,v_0)\in \mathcal{H}^{\vec{\sigma}}(\mathbb{T}^3)$. We extend $(u_0,v_0)$ to a periodic function $(\bar{u}_0,\bar{v}_0)$ on $\mathbb{R}^3$. Let $\eta\in C_c^\infty(\mathbb{R}^3)$ such that $\eta=1$ on $[-1,1]^3$, and define
\[\eta_T(x)=\eta(x/\langle T\rangle).\]
Here $\langle \cdot \rangle = (1+|\cdot|^2)^{1/2}$.

Consider the initial data problem on $\mathbb{R}^3$:
\begin{equation}
  \begin{cases}
\partial_t w-\Delta w+|w|^4w=0\\
(w,\partial_t w)|_{t=0}:=(\eta_T \bar{u}_0,\eta_T \bar{v}_0).
\end{cases}
\end{equation}
By \eqref{eqn: L5}, we have the estimate
\[\|(w,\partial_t w)\|_{C_t([0,T],\mathcal{H}^\sigma)}\le C(\|(w,\partial_t w)|_{t=0}\|_{\mathcal{H}^{\sigma}(\mathbb{R})})\]
Then, by \cite[Eqn. (3.7)]{OP}, we have, for $f\in H^s(\mathbb{T}^3)$
\begin{equation}
(1/C)\langle T\rangle^{3/2}\|f\|_{H^s(\mathbb{T}^3)}\le \|\eta_Tf\|_{H^s(\mathbb{R}^3)}\le  C\langle T\rangle^{3/2}\|f\|_{H^s(\mathbb{T}^3)}.
\end{equation}
It follows that
\[\|(w,\partial_t w)\|_{C_t([0,T],\mathcal{H}^\sigma)}\lesssim T^{3/2}.\]
By finite speed of propagation, for $|t|\le T$, the restriction of $(w(t),\partial_t w(t))$, to $[-\frac{1}{2},\frac{1}{2})^{3}\approx \mathbb{T}^3$ coincides with the solution $(u(t),v(t))$ of the quintic nonlinear wave equation on $\mathbb{T}^3$ with initial data $(u_0,v_0)$.

\end{document}